\documentclass[12pt,reqno]{amsart}

\usepackage{graphicx} 
\usepackage{amsmath, amsthm, amssymb,indentfirst}

\usepackage{amsmath, amsthm, amsopn, amssymb, enumerate}

\usepackage{amsmath , amssymb, latexsym }
\usepackage{graphics}
\usepackage{graphicx}
\usepackage{verbatim} 
\usepackage{color}
\usepackage{float}
\usepackage{pgfplots}
\def\Var{\textup{Var}}
\usepackage{url}

\pgfplotsset{
     standard/.style={
        axis x line=middle,
        axis y line=middle,
        every axis x label/.style={at={(current axis.right of origin)},anchor=north west},
        every axis y label/.style={at={(current axis.above origin)},anchor=north east}
    }
}

\pgfplotsset{width=13cm,compat=1.16}

\usepgfplotslibrary{fillbetween}
\usepackage[mathscr]{eucal}
\usepackage{pgf,tikz}
\usepackage{mathrsfs}
\usetikzlibrary{arrows}
\usepackage[font=small,labelfont=bf]{caption}
\usepackage{xifthen}
\usepackage{mathtools}
\theoremstyle{plain}
\newtheorem{theorem}{Theorem}[section]
\newtheorem{corollary}[theorem]{Corollary}
\newtheorem*{corollary*}{Corollary}
\newtheorem{prop}[theorem]{Proposition}
\newtheorem{lemma}[theorem]{Lemma}
\newtheorem*{proposition*}{Proposition}
\newtheorem*{theorem*}{Theorem}
\newtheorem*{lemma*}{Lemma}

\newtheorem*{claim*}{Claim}
\theoremstyle{definition}

\newtheorem*{definition*}{Definition}

\theoremstyle{remark}

\newtheorem*{obs*}{Observation}

\newtheorem{remark}{Remark}

\newcommand{\pr}[1]{\mathbb{P}\left(#1\right)}

\newcommand{\Bin}{\ensuremath{\mathrm{Bin}}}

\newcommand{\eps}{\varepsilon}
\newcommand{\Ex}[1]{\mathbb{E}\left[#1\right]}

\newcommand{\eq}[1]{\begin{equation}\label{eq:#1}}
\newcommand{\eqe}{\end{equation}}
\newcommand{\eqr}[1]{\eqref{eq:#1}}

\def\F{\mathcal{F}}

\def\HH{\mathcal{H}}

\def\RR{\mathbb{R}}

\def\eps{\varepsilon}
\def\le{\leqslant}
\def\ge{\geqslant}

\def\Bin{\textup{Bin}}

\def\<{\langle}
\def\>{\rangle}
\def\0{\textbf{0}}

\def\y{\mathbf{y}}

\newcommand{\rset}[1]{B_{#1}}

\newcommand{\countx}[3]{\ifthenelse{\isempty{#1}{}}{N^{#2}(\rset{#3})}{{N_{#1}^{#2}(\rset{#3})}}}

\newcommand{\inc}[3]{X_{#1}^{#2}(\rset{#3})}

\definecolor{redor}{rgb}{0.1,0.1,0.1}
\definecolor{sgreen}{rgb}{0.5, 1, 0.5} 

\definecolor{lblue}{rgb}{0.8, 0.82, 1} 
\definecolor{midblue}{rgb}{0.5, 0.5, 1} 
\definecolor{dblue}{rgb}{0.2, 0.2, 0.9}

\definecolor{lgr}{rgb}{0.8, 0.8, 0.8} 
\definecolor{purp}{rgb}{0.9, 0, 0.9} 
\definecolor{lred}{rgb}{1, 0.8, 0.8} 
\definecolor{dred}{rgb}{1, 0.5, 0.5} 

\definecolor{lgreen}{rgb}{0.8, 1, 0.8}
\definecolor{midgreen}{rgb}{0.3, 0.9, 0.3}

\definecolor{dgreen}{rgb}{0.1, 0.8, 0.1} 

\definecolor{dyellow}{rgb}{0.9, 0.9, 0} 

\definecolor{morange}{rgb}{1, 0.5, 0} 

\definecolor{turquoise}{rgb}{0, 0.8, 1}

\title[Deviation probabilities for arithmetic progressions]{Deviation probabilities for arithmetic progressions and irregular discrete structures}

\author{Simon Griffiths}
\address{Simon Griffiths, Departamento de Matem\'atica, PUC-Rio, Rua Marqu\^es de S\~ao Vicente 225, G\'avea, Rio de Janeiro 22451-900, Brazil}		
\email{simon@mat.puc-rio.br}

\author{Christoph Koch}
\address{Christoph Koch, Department of Statistics, University of Oxford, 24-29 St Giles', Oxford OX1 3LB, United Kingdom}
\email{christoph.koch@stats.ox.ac.uk}

\author{Matheus Secco}
\address{Matheus Secco, The Czech Academy of Sciences, Institute of Computer Science, Pod Vod\'{a}renskou v\v{e}\v{z}\'{\i} 2, 182 07 Prague, Czech Republic}
\email{matheussecco@gmail.com}

\thanks{\textbf{Acknowledgements:}. We would like to thank the organisers of the Workshop on structure and randomness in hypergraphs at the London School of Economics where this project began, in December 2018. We would like also to thank fellow participants Gonzalo Fiz Pontiveros and Matías Pavez Signé for their involvement in many fruitful discussions.  S.G. was supported by CNPq bolsas de produtividade em pesquisa (Proc. 310656/2016-8 and Proc. 307521/2019-2) and FAPERJ Jovem cientista do nosso estado (Proc. 202.713/2018).  C.K. was supported by EPSRC grant EP/N004833/1.  M.S. was supported during his PhD by the CAPES funding agency and subsequently by the Czech Science Foundation, grant number GJ20-27757Y, with institutional support RVO:6798580.  M.S. and C.K. thank PUC-Rio for its hospitality during their PhD/visit respectively.}

\begin{document}
\begin{abstract} 
Let the random variable $X\, :=\, e(\HH[B])$ count the number of edges of a hypergraph $\HH$ induced by a random $m$-element subset $B$ of its vertex set.  Focussing on the case that the degrees of vertices in $\HH$ vary significantly we prove bounds on the probability that $X$ is far from its mean.  It is possible to apply these results to discrete structures such as the set of $k$-term arithmetic progressions in the $\{1,\dots, N\}$.  Furthermore, our main theorem allows us to deduce results for the case $B\sim B_p$ is generated by including each vertex independently with probability $p$.  In this setting our result on arithmetic progressions extends a result of Bhattacharya, Ganguly, Shao and Zhao~\cite{BGSZ}.
We also mention connections to related central limit theorems.
\end{abstract}
\maketitle

\section{Introduction}

Let $W_3$ be the number of $3$-term arithmetic progressions in a random set $B_p\subseteq [N]=\{1,\dots ,N\}$ in which each element is included independently with probability $p$.  What can be said about the upper tail of the random variable $W_3$?  This question has been extensively studied in recent years and we shall use a discussion of this problem to motivate our results.

Given sequences $p=p_N$ and $\delta=\delta_N$ we may define a rate associated with the corresponding deviation by
\[
r(N,p,\delta)\, :=\, -\log\big(\pr{W_3\, \ge\, (1+\delta) \Ex{W_3}}\big)\, .
\]
The case that $\delta>0$ is a fixed constant is known as the \emph{large deviations regime}.  The asymptotic value 
\[
r(N,p,\delta)\, =\, (1+o(1))\delta^{1/2}p^{3/2}N\log(1/p)
\]
was obtained by Harel, Mousset and Samotij~\cite{HMS} for all $\delta>0$ and across the whole range of densities $N^{-2/3}(\log{N})^{2/3}\ll p\ll 1$.  This improved on earlier results of Bhattacharya, Ganguly, Shao and Zhao~\cite{BGSZ} (which covered the cases $p\ge N^{-1/36}(\log{N})^{1/3}$) and Warnke~\cite{WAP} (which gave the value of $r(N,p,\delta)$ up to a multiplicative constant).

We remark that Bhattacharya, Ganguly, Shao and Zhao~\cite{BGSZ} and Warnke~\cite{WAP} also obtained some moderate deviation results (with $\delta_N\to 0$).  In particular, Warnke~\cite{WAP} obtained the value of $r(N,p,\delta)$ up to a constant factor:
\[
r(N,p,\delta)\, =\, \Theta(1)\min\{\delta_N^{2}pN\, ,\, \delta_N^{2}p^{3}N^2 \, ,\, \delta_N^{1/2}p^{3/2}N\log(1/p)\}
\]
provided $\delta\ge N^{-1/18+\eps}$ for some $\eps>0$.  The three types of behaviour correspond to three \emph{regimes}.  We may define three regimes  \emph{Normal}, \emph{Poisson} and \emph{localised} according to which of $\delta_N^{2}pN\, ,\, \delta_N^{2}p^{3}N^2$ and $\delta_N^{1/2}p^{3/2}N\log(1/p)$ is minimal.  It is widely believed that
\[
r(N,p,\delta)\, =\, (1+o(1))\min\{3\delta_N^{2}pN/56(1-p)\, ,\, \delta_N^{2}p^{3}N^2/8 \, ,\, \delta_N^{1/2}p^{3/2}N\log(1/p)\}
\]
across the whole range of $p,\delta$ which correspond to moderate deviations in sparse random sets, i.e., $p\to 0$ and $\max\{1/\sqrt{pN},1/p^{3/2}N\}\ll \delta\ll1$, except possibly for hybrid behaviour near the borders between regimes.  Bhattacharya, Ganguly, Shao and Zhao~\cite{BGSZ} used a large deviation principle of Eldan~\cite{Eld}, which improved quantitively on~\cite{CD}, to prove this result for a part of the localised regime and part of the normal regime.  That is, they proved that $r(N,p,\delta)=\delta_N^{1/2}p^{3/2}N\log(1/p)$ provided $p\to 0$ and $\delta_N\gg p^{1/3}(\log(1/p))^{2/3}$, and $r(N,p,\delta) = (1+o(1))3\delta_N^{2}pN/56(1-p)$
provided $p \to 0$ and 
\[
p^{-3}N^{-1/6}(\log{N})^{7/6}\, ,\, p^{-1/2} N^{-1/12}(\log{N})^{7/12}\, \le \, \delta_N\,\ll\,  p^{1/3}(\log(1/p))^{2/3}\, .
\] 
In this context our contribution is to prove the same asymptotic value provided $\sqrt{\log{N}/pN} \ll \delta_N \ll p^{1/2}$.  

The following figure illustrates these regions.  We consider here the cases $p = N^{\gamma}$ and $\delta_N = N^{\theta}$, where $\gamma, \theta \le 0$.

We remark that the form of the expression for $r(N,p,\delta)$ is given by $(1+o(1))\delta_N^2\Ex{W_3}^2/2\Var(W_3)$ in both the Normal and Poisson regimes.  The only distinction is that the leading term in the variance is $(1+o(1))7p^5N^3(1-p)/12$ in the Normal regime and $(1+o(1))\Ex{W_3}=(1+o(1))p^3N^2/4$ in the Poisson regime.

Our results also apply to the lower tail.  It would be of interest to determine the asymptotic rate for the lower tail in general.  The order of magnitude $\Theta(1)\min\{p^3N^2,pN\}$ may be obtained using Harris’ inequality~\cite{Harris} and Janson’s inequality, see Theorem 2.14 of~\cite{JLR}.  See~\cite{JW2015} and~\cite{MNPS} for more detailed discussions of lower tail problems.

\begin{figure}[H]
\centering  
  \resizebox{16cm}{!}{
\begin{tikzpicture}
\begin{axis}[unit vector ratio={7 5}, xmin=63,ymin=38,xmax=135,ymax=138,axis lines=middle,
    axis line style={->},
   standard,
    xlabel={\tiny $\theta$},
    ylabel={\tiny $\gamma$},
    xtick={63.1,94.5,119,126},
    xticklabels={\hspace{2mm}{\tiny $-1/2$}, {\tiny $-1/4$}, {\tiny $-1/18$}, {\tiny $0$}},  
    ytick={42.1,63,84,105,126},    
     yticklabels={{\tiny $-2/3$},{\tiny $-1/2$},{\tiny $-1/3$},{\tiny $-1/6$},{\tiny $0$}},auto]

        \begin{scope}                                                 
  \clip  (94.5,63) -- (105,63) --  (126,42) -- cycle;
\foreach \y in {16, 19,...,80} {\addplot[ultra thick,red] coordinates{(90,\y) (126,30+\y)};}
\foreach \y in {17, 20,...,80} {\addplot[ultra thick,lred] coordinates{(90,\y) (126,30+\y)};}
\foreach \y in {18, 21,...,80} {\addplot[ultra thick,dred] coordinates{(90,\y) (126,30+\y)};}
\end{scope}    

  \begin{scope}    
      \clip (123,120) -- (94.5,63) -- (105,63) --  (123.9,119.7) -- cycle;
\foreach \y in {18, 21,...,120} {\addplot[ultra thick,blue] coordinates{(90,\y) (126,30+\y)};}
\foreach \y in {19, 22,...,120} {\addplot[ultra thick,lblue] coordinates{(90,\y) (126,30+\y)};}
\foreach \y in {20, 23,...,120} {\addplot[ultra thick,midblue] coordinates{(90,\y) (126,30+\y)};}
\end{scope}   

       \begin{scope}    
      \clip (126,119) -- (123.9,119.7) -- (105,63) -- (126,42) -- cycle;
\foreach \y in {11, 14,...,120} {\addplot[ultra thick,midgreen] coordinates{(90,\y) (126,30+\y)};}
\foreach \y in {12, 15,...,120} {\addplot[ultra thick,lgreen] coordinates{(90,\y) (126,30+\y)};}
\foreach \y in {13, 16,...,120} {\addplot[ultra thick,sgreen] coordinates{(90,\y) (126,30+\y)};}

\end{scope}

                   \addplot[name path = Sigma1, line width=2pt, gray] coordinates{(94.5,63) (63,126)};
                \addplot[name path = Sigma2, line width=2pt, gray] coordinates{(94.5,63) (126,42)};
                                \addplot[name path = Sigma2plus, line width=2pt, gray] coordinates{(94.35,63.1) (126,42)};

         \addplot[name path = Local1] coordinates{(126,126) (123.9,119.7) (105,63)};
     \addplot[name path = Local2] coordinates{(105,63) (126,42)};
     
     \addplot[name path = OurR] coordinates{(115.5,126) (117.6,121.8) (123,120) (94.5,63)};

       \addplot[name path = Zhao1] coordinates{(126,126) (123,120) (123.9,119.7)};
              \addplot[name path = Z1] coordinates{(126,126) (123.9,119.7)};
        \addplot[name path = Zhao2] coordinates{(126,126) (123,120) (117.6,121.8) (115.5,126)};
                \addplot[name path = Zhao3] coordinates{(126,126) (126,119) (123.9,119.7)};

             \addplot[name path = NP] coordinates{(105,63) (94.5,63)};

             \addplot[name path = TopL] coordinates{(115.5,126) (63,126)};
                          \addplot[name path = TopR] coordinates{(115.5,126) (126,126)};
                          
                                       \addplot[name path = R, line width=2pt, dgreen] coordinates{(126,126) (126,119)};
                                                                              \addplot[name path = R, line width=2pt, green] coordinates{(126,119) (126,42)};

                                           \addplot[line width=4pt,color = morange,  domain = 119:126]{39};


\addplot[lblue] fill between[of = Sigma1 and OurR];
\addplot[turquoise] fill between[of = TopR and Zhao2];
\addplot[dblue] fill between[of = Zhao1 and Z1];
\addplot[dgreen] fill between[of = Zhao3 and Z1];


\end{axis}

\end{tikzpicture}
}
\caption{The grey lines on the left represent deviations of the order of the standard deviation and so are covered by the central limit theorem~\cite{Koch}.  The blue areas ought to belong to the ``Normal'' regime in which we would expect that $r(N,p,\delta)$ to be of the form $(1+o(1))3\delta_N^2pN/56(1-p)$.  Our results show this in the light blue and turquoise regions, while Bhattacharya, Ganguly, Shao and Zhao~\cite{BGSZ} covered the turquoise and dark blue regions.  The question remains open in the striped regions.  In the red striped region, the ``Poisson" regime, we would expect $r(N,p,\delta)=(1+o(1))\delta_N^{2}\Ex{W_3}/2=(1+o(1))\delta_N^2p^3N^2/8$.  In the green striped region, the ``localisition'' regime, we would expect $r(N,p,\delta)=(1+o(1))2\delta_N^{1/2}\Ex{W_3}^{1/2}\log(1/p)=(1+o(1))\delta_N^{1/2}p^{3/2}N\log(1/p)$.  The known cases in this regime are also due to Bhattacharya, Ganguly, Shao and Zhao~\cite{BGSZ} and are represented by the dark green triangle.  The right hand side of the figure represents large deviations which have been resolved by Harel, Mousset and Samotij~\cite{HMS} as discussed above.  We have not included Warnke's moderate deviation results (which give the value of $r(N,p,\delta)$ up to a multiplicative constant) in the figure, however we have marked the range of $\delta$ covered by his results using an orange line on the $\theta$-axis.}
\end{figure}
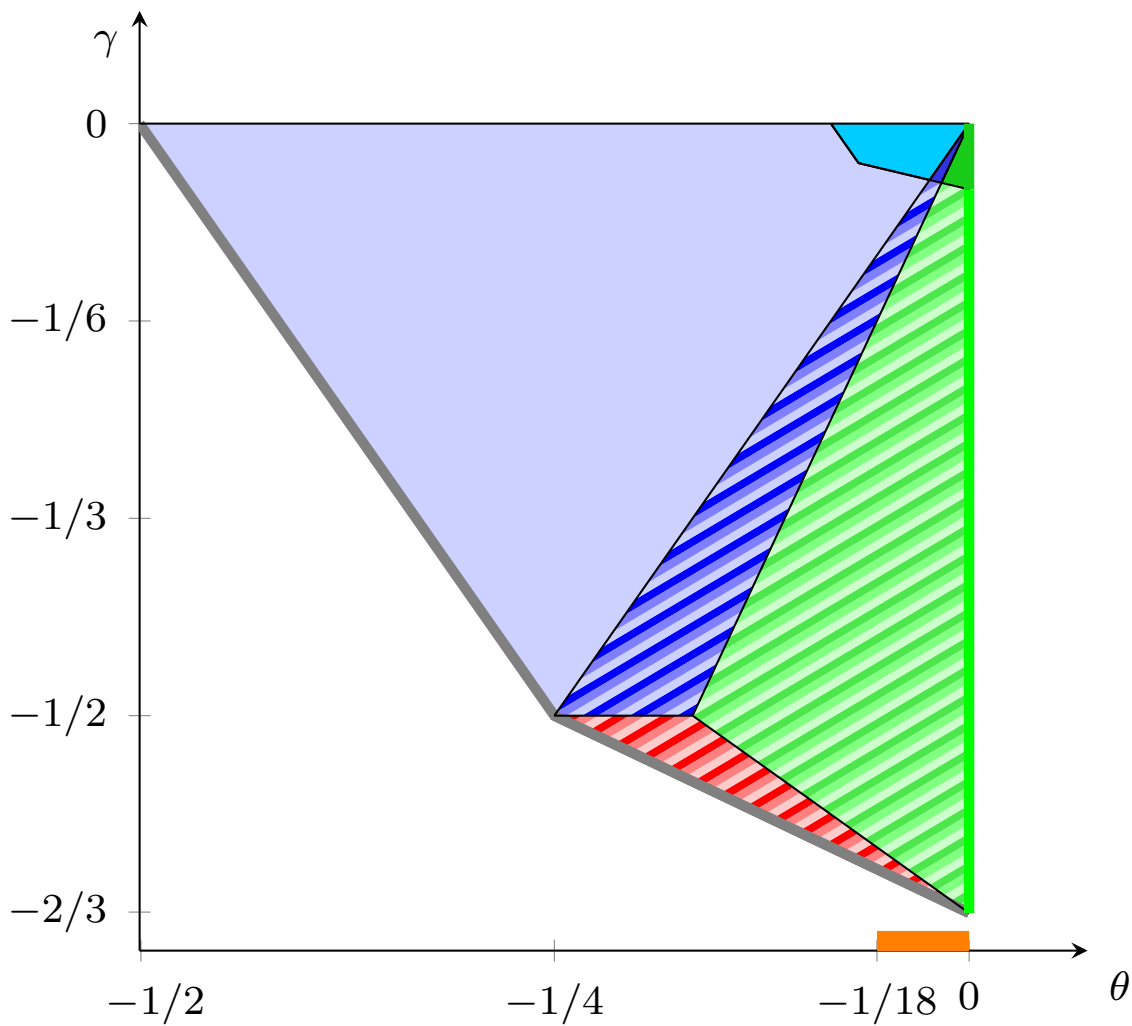

%
%

\textbf{Moderate deviations in the model $B_p$ in general}

We shall prove results analogous to those discussed above in a \emph{much} more general setting.  Given a $k$-uniform hypergraph $\HH$ with vertex set $[N]$ we may define the random variable
\[
X\, :=\, e(\HH[B_p])
\]
which counts the number of edges of a hypergraph $\HH$ in the random subset $B_p\subseteq [N]$.
Many results on deviation probabilities apply in this setting.  In particular, the incredibly useful and versatile inequality of Kim and Vu~\cite{KV2000} is now a fundamental tool of probabilistic combinatorics.  

Many results have focussed on large deviations.  Janson and Ruci\'nski~\cite{JR} determined (under certain conditions) the log probability $\log(\pr{X>(1+\delta)\Ex{X}})$ up to a factor of order $\log(1/p)$.  Warnke~\cite{Wlog} determined (under certain conditions) the log probability up to a constant factor.  And in recent work Bhattacharya and Mukherjee~\cite{BM2019} proved that a framework for studying large deviations, introduced by Chatterjee and Varadhan~\cite{CV2011} (in the context of subgraph counts), may also be used in this setting, and serves as a basis for analysing questions such as symmetry breaking. 

Our results in this setting depend on some basic parameters of the hypergraph.  Let us set 
\[
N^{\HH}(B)\, :=\, e(\HH[B])\, ,\phantom{\Big|}
\]
the number of edges of $\HH$ in the set $B$, and $D^{\HH}(B_p)=N^{\HH}(B_p)-p^ke(\HH)$, the deviation of $N^{\HH}(B_p)$ from its mean.  For a vertex $i\in [N]$ we write $d_{\HH}(i)$ for the degree of $i$ in $\HH$ so that
\[
\bar{d}(\HH)\, :=\, \frac{1}{N}\sum_{x=1}^{N} d_{\HH}(x)\, =\, k e(\HH)/N
\]
and
\[
\sigma^2(\HH)\, =\, \frac{1}{N}\sum_{x=1}^{N} (d_{\HH}(x)\, -\, \bar{d}(\HH))^2
\]
are the average degree and degree variance respectively.

We may extend the definition of degree to sets, so that $d_{\HH}(R)$ is the number of edges containing the set $R$.  We then define $\Delta_r(\HH)$ to be the maximum $r$-degree of $\HH$, i.e., 
\[
\Delta_r(\HH)\, :=\, \max\{d_{\HH}(R)\, :\, |R|=r\}\, .
\]
Our conditions on $\HH$ will simply be that, for some $r\ge 2$, we have $\Delta_r(\HH)=O(1)$  and $\sigma^2(\HH)=\Omega(N^{2r-2})$.  In particular this asks that the standard deviation of degrees be of the same order $N^{r-1}$ as the maximum degree (as $\Delta_1(\HH)\le N^{r-1}\Delta_r(\HH)=O(N^{r-1})$).

We may now state our main result in this setting.  In fact it makes no difference to the proof to state the result for weighted hypergraphs in which a positive weight is associated to each edge.  Naturally, all parameters count edges with weights.  For example $e(\HH)$ then refers to the sum of all weights and a degree $d_{\HH}(x)$ refers to the sum of weights of edges containing $x$.

\begin{theorem}\label{thm:pworld} 
Let $2\le r\le k$ and $C$ be integers.  Let $\HH_N$ be a sequence of (weighted) $k$-uniform hypergraphs with $V(\HH_N) = [N]$, $\Delta_r(\HH_N)\le C$ and $\sigma^2(\HH_N)\ge N^{2r-2}/C$ for all $N$.  Given a sequence $p=p_N$ which may be a constant in $(0,1/2)$ or converge to $0$, let $\delta_N$ be a sequence such that
\[
\sqrt{\frac{\log{N}}{pN}}\, \ll\, \delta_N\, \ll\, p^{(k-r)/2(r-1)}\, .
\]
Then
\[
\pr{ D^{\HH_N}(B_p)\, \ge\, \delta_N p^k e(\HH_N)} \,=\,  \exp\left(\frac{-(1+o(1))\delta_N^2 pe(\HH_N)^2}{2(1-p)\big(\bar{d}(\HH_N)^2+\sigma^2(\HH_N)\big)N}\right)\, .
\]
Furthermore the same result holds for the lower tail probability $\pr{ D^{\HH_N}(B_p)\, \le\, -\delta_N p^k e(\HH_N)}$.
\end{theorem}

In fact we shall work almost exclusively in another model $B_m$, a uniformly random $m$-element subset of $[N]$.  Theorem~\ref{thm:pworld} will be a consequence of a similar theorem for the $B_m$ model, see Theorem~\ref{thm:main}, together with standard results about the tail of the binomial distribution.

\textbf{Moderate deviations in the model $B_m$ in general}

We now consider the uniform model $B\sim B_m$, where $B_m$ is a uniformly random $m$-element subset of $[N]$.  It has been less common to work with this model, however we remark that Warnke~\cite{Wlog} shows that his results also hold in the $B_m$ model.  Furthermore, moderate deviations were studied recently in the $B_m$ model for hypergraphs which are regular or close to regular~\cite{FGSS}.

We define
\[
L^{\HH}(m)\, :=\, \Ex{N^{\HH}(B_m)}\, ,
\]
the expected value of $N^{\HH}(B_m)$ in the model $B\sim B_m$.  And set
 \[
 D^{\HH}(B_m)\, :=\, N^{\HH}(B_m)\, -\, L^{\HH}(m)\, 
 \]
to be the deviation of $N^{\HH}(B_m)$ from its mean.

We set $t:=m/N$ to be the density of the random set.  The result makes sense with $t\in (0,1/2]$ a constant or a sequence $t=t_N\in (0,1/2]$.   

\begin{theorem}\label{thm:main}
Let $2\le r\le k$ and $C$ be integers.  Let $\HH_N$ be a sequence of (weighted) $k$-uniform hypergraphs with $V(\HH_N) = [N]$, $\Delta_r(\HH_N)\le C$ and $\sigma^2(\HH_N)\ge N^{2r-2}/C$ for all $N$.  Let $m/N =t \le 1/2$.  Let $a_N$ be a sequence such that
\[
t^{k-1/2}N^{r-1/2}(\log{N})^{1/2}\, \ll\, a_N\, \ll\, t^{k-1/2+(k-1)/2(r-1)}N^r \, .
\]
Then
\[
\pr{D^{\HH_N}(B_m)\,  \ge\, a_N}\, = \, \exp \left(  \frac{-(1+o(1))a_N^2}{2(1-t)t^{2k-1}\sigma^2(\HH_N) N} \right)\, .
\]
Furthermore the same holds for the lower tail probability $\pr{D^{\HH_N}(B_m)\,  \le\, -a_N}$.
\end{theorem}

\begin{remark} It is natural to ask whether the interval of deviations is best possible.  The lower bound is necessary, except perhaps for the $(\log{N})^{1/2}$, as $t^{k-1/2}N^{r-1/2}$ is the order of the standard deviation.  The upper bound is essentially best possible in the case $r=k$, but not for $r<k$.  In particular, we believe that it may be extended by an extra poly log factor by making changes to Section~\ref{sec:known}, for example, using Warnke's inequality~\cite{Wlog} instead of using a moment bound similar to that of Janson and Ruci\'nski~\cite{JR}.  However, it may well be possible to extend the interval even further, perhaps as far as $t^{r(2k-1)/(2r-1)}N^r$, which would be best possible (up to polylog terms).
\end{remark}

\textbf{A discussion of the proof of Theorem~\ref{thm:main}}

It is clear that the variance of the degrees plays a central role.  In particular, the more variance the more likely deviations become.  This is because deviations are driven by the degrees of the vertices selected.  It may seem counterintuitive that deviations of the random variable $D^{\HH}(B_m)$ which in principle could depend on many properties of the hypergraph $\HH$ in the end depend almost entirely on the degree distribution of $\HH$.  However we will make this assertion concrete by showing that $D^{\HH}(B_m)$ is generally very well approximated by a process which considers only the degrees.  

Let us now define this ``degree'' process.  We may suppose that $B_m$ is generated as $B_m=\{b_1,\dots ,b_m\}$ where $b_1,\dots ,b_N$ is a uniformly random permutation of the elements of $[N]$.  We set $t:=m/N$ and $s:=i/N$.  Let us define
\[
\Lambda^{\HH}(B_m)\, :=\, \sum _{i=1}^{m} \frac{t^{k-1}(1-t)}{1-s}\, X^{\HH}_1(B_i)
\]
where $X^{\HH}_1(B_i)$ is defined to be the deviation of the degree of the $i$ th element $b_i$ from its conditional expectation, i.e., 
\[
X^{\HH}_1(B_i)\, :=\, d_{\HH}(b_i)\, -\, \Ex{d_{\HH}(b_i)|B_{i-1}} \, .
\]

The proof of Theorem~\ref{thm:main} now partitions naturally into two tasks.  First, we must show that $D^{\HH}(B_m)$ is generally very close to $\Lambda^{\HH}(B_m)$.  This is achieved by considering the full martingale representation for $D^{\HH}(B_m)$ given in Section~\ref{sec:Mart} (and previously in~\cite{FGSS}) and showing that all the terms $X^{\HH}_{\ell}(B_i)$ which occur are predictable in terms of $X^{\HH}_1(B_i)$.  It is in proving this approximation that we make use of previous large deviation results, in particular we use/adapt the result of Janson and Ruci\'nski~\cite{JR}.

The second task is to prove a deviation result for $\Lambda^{\HH}(B_m)$ (Proposition~\ref{prop:main}).  We do so using Freedman's inequalities for martingale deviations.  This turns out to be relatively straightforward.  We simply need to control the contributions $\Ex{X^{\HH}_1(B_i)^2|B_{i-1}}$ to the quadratic variation (see Section~\ref{sec:var}).

\textbf{Returning to arithmetic progressions}

Let us now state more formally our result for $3$ term arithmetic progressions.  

Let $\HH^{3}$ be the $3$-uniform hypergraph with vertex set $[N]$ and edges corresponding to $3$-APs in $[N]$. It follows from simple calculations (see~\cite{SeccoThesis}) that
\begin{align*}
\bar{d}(\HH^{3}) &= (1+o(1))\frac{3N}{4} \, , \\ 
\sigma^2(\HH^{3}) &= (1+o(1))\frac{N^2}{48} \, , \\
 e(\HH^{3}) &= (1+o(1))\frac{N^2}{4}.
\end{align*}
It is also easy to verify that $\Delta_2(\HH^{3}) = O(1)$ and so we can apply Theorem~\ref{thm:pworld} with $r = 2$ to obtain the following result.

\begin{corollary}\label{cor:3APnonreg}
Given a sequence $p = p_N$ which may be a constant in $(0,1/2)$ or converge to $0$, let $\delta_N$ be a sequence satisfying
\[
\sqrt{\frac{\log N}{pN}}\, \ll\, \delta_N\, \ll\, p^{1/2}.
\]
Then
\begin{equation}\label{eq:3APnonregp}
\pr{ D^{\HH^3}(B_p)\, \ge\, \delta_N p^3N^2/4}\, = \, \exp \left( \frac{-(3+o(1))\delta_N^2pN}{56(1-p)} \right).
\end{equation}
\end{corollary}

Results for longer arithmetic progressions and another example (additive quadruples) are given in Section~\ref{sec:appl}.

\textbf{Central limit theorems}

Let us mention that our results also give central limit theorems for $D^{\HH}(B_m)$ and $D^{\HH}(B_p)$.  We remark that the central limit theorem (and even a bivariate version) is already known for the special case of arithmetic progressions~\cite{Koch}.  However, we are not aware of a central limit theorem in the more general context of random induced subhypergraphs.  We state first our result for $D^{\HH}(B_m)$.

\begin{theorem}\label{thm:clt} 
Let $2\le r\le k$ and $C$ be integers.  Let $\HH_N$ be a sequence of (weighted) $k$-uniform hypergraphs with $V(\HH_N) = [N]$, $\Delta_r(\HH_N)\le C$ and $\sigma^2(\HH_N)\ge N^{2r-2}/C$ for all $N$.  Suppose that $t=m/N$ satisfies $(\log{N}/N)^{(r-1)/(k-1)}\ll t\, \le 1/2$.
Then
\[
\frac{D^{\HH_N}(B_m)}{(1-t)^{1/2}t^{k-1/2}\sigma(\HH_N) N^{1/2}} \, \longrightarrow \, N(0,1)
\]
in distribution.
\end{theorem}
 
Since it is not the focus of the article let us mention now how one may prove the theorem using results of this paper.  Classic papers on the Martingale central limit theorem include~\cite{Brown},~\cite{HBrown} and~\cite{Dvoretzky}.  Although it is also worth consulting more recent results such as~\cite{Mourrat}, and the references therein.
  
\begin{proof}[Proof of Theorem~\ref{thm:clt}] Let $\sigma_1^2=(1-p)p^{2k-1}\sigma^2(\HH_N)N$.  One may easily prove that $\Lambda^{\HH_N}(B_m)/\sigma_1$ converges in distribution to a standard Gaussian using some version of the martingale central limit theorem together with Proposition~\ref{prop:var} to control the quadratic variation of the process.  The result then follows immediately from the fact that the difference between $D^{\HH}(B_m)/\sigma_1$ and $\Lambda^{\HH}(B_m)/\sigma_1$ converges to $0$ in probability by Proposition~\ref{prop:approx}.
\end{proof}

One may then easily deduce the corresponding result for the binomial $B_p$ model.

\begin{theorem}\label{thm:pworldclt} 
Let $2\le r\le k$ and $C$ be integers.  Let $\HH_N$ be a sequence of (weighted) $k$-uniform hypergraphs with $V(\HH_N) = [N]$, $\Delta_r(\HH_N)\le C$ and $\sigma^2(\HH_N)\ge N^{2r-2}/C$ for all $N$.  Then for every sequence $p=p_N$ which may be a constant in $(0,1/2)$ or converge to $0$ while satisfying $p\gg (\log{N}/N)^{(r-1)/(k-1)}$, we have 
\[
\frac{D^{\HH_N}(B_p)}{(1-p)^{1/2}p^{k-1/2}\big(\bar{d}(\HH_N)^2+\sigma^2(\HH_N)\big)^{1/2}N^{1/2}}\, \longrightarrow \, N(0,1)
\]
in distribution.
\end{theorem}

\begin{proof} 
Let $\sigma_1^2=(1-p)p^{2k-1}\sigma^2(\HH_N)N$ and $\sigma_2^2=(1-p)p^{2k-1}\bar{d}(\HH_N)^2N$. By standard properties of the normal distribution it suffices to show that $D^{\HH_N}(B_p)$ may be expressed as a sum $X_N+X'_N$ where $X_N/\sigma_1, X'_N/\sigma_2\to N(0,1)$ in distribution and $X_N$ is orthogonal to $X'_N$ in the sense that for all $m$ we have $\Ex{X_N|X'_N=m}=0$.  It is easy to find such a representation, defining the random variable $M=|B_p|$, the number of elements in the random set $B_p$, we may simply decompose $D^{\HH}(B_p)$ as follows
\begin{align*}
D^{\HH}(B_p)\, &=\, N^{\HH}(B_p)\, -\, p^{k}e(\HH)\phantom{\Big|}\\
& =\, \big(N^{\HH}(B_p)\, -\, \Ex{N^{\HH}(B_p)|M}\big)\, +\, \big(\Ex{N^{\HH}(B_p)|M}\, -\, p^ke(\HH)\big) \, .\phantom{\Big|}
\end{align*}
The first bracketed quantity is exactly $D^{\HH}(B_m)$ for the corresponding value $m$ of $M$, it is easily verified (using Theorem~\ref{thm:clt}) that dividing by $\sigma_1$ this converges in distribution to a standard Gaussian.   The second bracketed quantity corresponds to $L^{\HH}(M)-p^{k}e(\HH)$.  Using the fact that $M\sim Bin(N,p)$ is asymptotically normally distributed it is easily verified that $(L^{\HH}(M)-p^{k}e(\HH))/\sigma_2$ also converges in distribution to a standard Gaussian, as required.
\end{proof}

\textbf{Notation} 

Let us emphasise that $t$ denotes $m/N$, the density of the random set $B_m\subseteq [N]$, throughout the article.  Likewise $s$ denotes $i/N$.

\textbf{Overview of the paper}

In Section~\ref{sec:Aux} we give important auxiliary results and well known inequalities which will be useful throughout the article.   In Section~\ref{sec:known} we adapt to our context the approach of Janson and Ruci\'nski~\cite{JR} to bound large deviations.  And at the end of Section~\ref{sec:known} we state a corollary of these results for link hypergraphs.  Using this corollary we show in Section~\ref{sec:ellto1} that $X^{\HH}_{\ell}(B_i)$ may be well approximated by a multiple of $X^{\HH}_1(B_i)$.  This result will be sufficient to show that $D^{\HH}(B_m)$ is well approximated by $\Lambda^{\HH}(B_m)$ in Section~\ref{sec:approx}.  This represents the first major task in proving Theorem~\ref{thm:main}.  

In Section~\ref{sec:var} we study the quadratic variation of the martingale representation of $\Lambda^{\HH}(B_m)$.  Finally, in Section~\ref{sec:main} we prove deviation bounds for $\Lambda^{\HH}(B_m)$, making use of our control of the quadratic variation, and we deduce Theorem~\ref{thm:main}.  In Section~\ref{sec:pworld} we show how Theorem~\ref{thm:pworld} for the model $B\sim B_p$ follows. In Section~\ref{sec:appl} we give applications of Theorems~\ref{thm:main} and~\ref{thm:pworld} to arithmetic progressions and solutions of the Sidon equation in $\{1,\dots ,N\}$.

\section{Preliminaries and tools}\label{sec:Aux}

In this section we lay out some preliminary results which we will rely on in the rest of the paper.  In Section~\ref{sec:Mart} we state a martingale representation for the deviation $D^{\HH}(B_m)$ given first in~\cite{FGSS} (see also~\cite{GGS2019} for a similar result in the context of subgraph counts).

\subsection{Martingale representation}\label{sec:Mart}

In order to state the result we require some notation.  Let us define $B_m=\{b_1,\dots ,b_m\}$ where $b_1,\dots ,b_N$ is a uniformly random permutation of $[N]$.  We may now count the evolution of the number of edges of $\HH$ contained in $B_i$ for $i\le m$.  In fact it is useful to also consider partially filled edges.  For $1\le \ell\le k$ we let $N^{\HH}_{\ell}(B_m)$ be the number of pairs $(S,f)$ such that $|S|=\ell, f\in E(\HH)$ and $S\subseteq f\cap B_i$.  That is $N^{\HH}_{\ell}(B_m)$ counts (with multiplicity) the number of $\ell$-element subsets of $B_i$ contained in edges of $\HH$.

We shall study the one step increment of this quantity.  Set
\[
A^{\HH}_\ell(B_i)\, :=\, \big|\{(S,f)\, :\, |S|=\ell,\, b_i\in S, \, S\subseteq f\cap B_{i}\}\big|\, =\, N^{\HH}_{\ell}(B_i)\, -\, N^{\HH}_{\ell}(B_{i-1})\, .
\]
Finally we may define centered versions of these increments:
\[
X^{\HH}_{\ell}(B_i)\, :=\, A^{\HH}_\ell(B_i)\, -\, \Ex{A^{\HH}_\ell(B_i)|B_{i-1}}\, .
\]
We may now state the martingale representation of $D^{\HH}(B_m)$ given in~\cite{FGSS}.

\begin{lemma}\label{lem:Mart}
	Let $\HH$ be a $k$-uniform hypergraph on $[N]$. Then
	\begin{equation}\label{eq:Mart}
	D^{\HH}(B_m)\, =\, \sum_{i=1}^{m} \sum_{\ell=1}^{k} \frac{(N-m)_{\ell}(m-i)_{k-\ell}}{(N-i)_k}\, \inc{\ell}{\HH}{i}\, .
	\end{equation}
\end{lemma} 

%
%
%
%

\subsection{Probability inequalities}

%

We now state some auxiliary probability inequalities for martingale deviations, namely the classical Hoeffding-Azuma martingale inequality, Freedman's inequality and its converse.  We begin with the Hoeffding-Azuma inequality~\cite{Azuma,Hoeff}.

\begin{lemma}[Hoeffding--Azuma inequality]\label{lem:HA}
Let $(S_i)_{i=0}^{m}$ be a martingale with increments $(X_i)_{i=1}^{m}$, and let $c_i=\|X_i\|_{\infty}$ for each $1\le i\le m$.
Then, for each $a>0$, 
\[
\pr{S_m-S_0\, >\, a}\, \le \, \exp\left(\frac{-a^2}{2\sum_{i=1}^{m}c_i^2}\right)\, .
\]
Furthermore, the same bound holds for $\pr{S_m-S_0\, <\, -a}$.
\end{lemma}

Probabilistic intuition would suggest that variance (or rather quadratic variation) should be more relevant than $\sum_{i}\|X_i\|_{\infty}^2$.  Freedman's inequality~\cite{F} allows us to replace $\sum_{i}\|X_i\|_{\infty}^2$ by (essentially) the quadratic variation of the process up to that point plus a term which is often negligible.

\begin{lemma}[Freedman's inequality]\label{lem:Freedman}
Let $(S_i)_{i=0}^{m}$ be a martingale with increments $(X_i)_{i=1}^{m}$ with respect to a filtration $(\F_i)_{i=0}^{m}$, let $R \in \RR$ be such that $\max_i |X_i| \le R$ almost surely, and let
\[
V(j) := \sum_{i=1}^{j}  \Ex{ X_i^2 | \F_{i-1} }.
\]
Then, for every $\alpha, \beta > 0$, we have

\[
\pr{S_j - S_0 \ge \alpha \mbox{     and     } V(j) \le \beta \mbox{     for some j}} \le \exp \left( \frac{-\alpha^2}{2(\beta + R\alpha)} \right).
\]
\end{lemma}

Freedman also proved a converse for this inequality, which requires some new notation. Define the stopping time $\tau_{\alpha}$ to be the least $j$ such that $S_j > S_0 + \alpha$ and let $T_{\alpha} := V(\tau_{\alpha})$.

The following is the converse of Freedman's inequality~\cite{F}.

\begin{lemma}[Converse Freedman inequality]\label{lem:Freedman_converse}
Let $(S_i)_{i=0}^{m}$ be a martingale with increments $(X_i)_{i=1}^{m}$ with respect to a filtration $(\F_i)_{i=0}^{m}$, let $R \in \RR$ be such that $\max_i |X_i| \le R$ almost surely, and let $T_{\alpha}$ be as defined above. Then, for every $\alpha, \beta > 0$, we have
\[
\pr{T_{\alpha} \le \beta} \ge \frac{1}{2} \exp \left( \frac{-\alpha^2(1+4\delta)}{2\beta}\right),
\]
where $\delta$ is minimal such that $\beta/\alpha \ge 9R\delta^{-2}$ and $\alpha^2/\beta \ge 16\delta^{-2}\log (64\delta^{-2})$.
\end{lemma}

\section{A large deviations bound and application to link hypergraphs}\label{sec:known}

The main aim of this section is to prove a bound on the probability of a deviation in link hypergraphs $\HH(x)$ of a hypergraph $\HH$, see Corollary~\ref{cor:link}.  Since $\HH(x)$ is in some sense ``just some hypergraph'' we obtain the bound by simply applying the following proposition which bounds the probability of large deviations.  The proposition is proved using Janson's inequality for the lower tail and a moment argument similar to that of Janson and Ruci\'nski ~\cite{JR} for the upper tail.

It will be useful to allow the edges of the hypergraph to have a non-negative weight.  Given such a weighted hypergraph we naturally include the weight in all related parameters, so that $e(\HH)$ denotes the sum of all weights, the degree of a vertex is interpreted as the sum of weights of edges containing that vertex, and, for example $N^{\HH}(B_m)$ is the sum of the weights of edges of $\HH$ contained in $B_m$.

\begin{prop}\label{prop:known}  Let $1\le r\le k$ and $C$ be integers and let $\eps>0$, then there is a constant $c=c(k,C,\eps)>0$ such that the following holds.  Let $\HH$ be a $k$-uniform (weighted) hypergraph with $\Delta_r(\HH)\le C$.  Let $t\ge N^{-r/k}$, and let $m=tN$.  Then, for all $i\le m$, we have
\[
\pr{|D^{\HH}(B_i)|\, \ge\, \eps t^k N^r}\, \le\, 2N\exp(-ct^{k/r}N)\, .
\]
\end{prop}

\begin{remark} A stronger bound with $t^{k/r}N$ replaced by $\min\{t^kN^r,t^{k/r}N\log{N}\}$ ought to follow by instead using Warnke's results~\cite{Wlog} to bound the upper tail.  Doing so would allow one to extend the intervals of deviations considered in Theorem~\ref{thm:main} and Theorem~\ref{thm:pworld} by a polylog factor.
\end{remark}

We now state a lemma which bounds moments of the count of edges of $\HH$ in the random set.  In fact we work in the $B_p$ model, in which each elements of $[N]$ is included in $B_p$ independently with probability $p$.  The same result holds in the $B_m$ model but it is marginally easier to prove in the $B_p$ model.  The result is different only in a minor way from the bound proved by Janson and Ruci\' nski~\cite{JR}, and we use essentially the same proof.

\begin{lemma}\label{lem:moment} Let $1\le r\le k$ and $C$ be integers and let $\eps>0$, then there is a constant $c=c(k,C,\eps)>0$ such that the following holds.  Let $\HH$ be a $k$-uniform (weighted) hypergraph with $\Delta_r(\HH)\le C$.  Let $t\in (0,1)$.  Then, for all $p\le t$, we have
\[
\Ex{N^{\HH}(B_p)^{\ell}}\, \le\, \max\left\{\mu^{\ell}(1+\eps/2)^\ell\, , \, \left(\frac{\eps t^kN^r}{2}\right)^\ell\right\}\,
\]
for all positive integers $\ell\le ct^{k/r}N$, where $\mu:=p^{k}e(\HH)$.
\end{lemma}

\begin{proof}  We begin by proving that for any set $U$ of at most $k\ell$ elements of $[N]$ we have
\eq{begin}
\sum_{e\in E(\HH)}\pr{e\subseteq B_p|U\subseteq B_p}\, \le\, \max\left\{\mu(1+\eps/2)\, , \, \frac{\eps t^kN^r}{2}\right\}
\eqe
provided $\ell\le ct^{k/r}N$, for a sufficiently small small constant $c=c(k,C,\eps)>0$.

Let us count how many edges of $\HH$ have certain intersections with the set $U$.  We begin by bounding the number of edges $e$ with $|e\cap U|\ge r$.  By our condition that $\Delta_r\le C$, there are clearly at most $C\binom{k\ell}{r}\, \le\, Ck^r\ell^r$ such edges.  For $1\le j<r$ we have $\Delta_j\le N^{r-j}
\Delta_r\le CN^{r-j}$ and so there are at most $\binom{k\ell}{j}CN^{r-j}\le Ck^j\ell^jN^{r-j}$ edges $e$ with $|e\cap U|=j$.  Finally, we simply use $e(\HH)$ as the upper bound for the number of edges not intersecting $U$.

Armed with these bounds we have that
\begin{align*}
\sum_{e\in E(\HH)}\pr{e\subseteq B_p|U\subseteq B_p}\, &\le\, e(\HH)p^k\, +\, \sum_{j=1}^{r-1}Ck^j\ell^jN^{r-j}p^{k-j}\, +\, Ck^r\ell^r\\
& \le \, e(\HH)p^k\, +\, Ck^r t^k N^r \sum_{j=1}^{r-1}\left(\frac{\ell}{tN}\right)^j\, +\, Ck^r\ell^r\\
&\le\, e(\HH)p^k\, +\, Ck^r t^k N^r \sum_{j=1}^{r-1}c^j\, +\, Ck^rc^r t^kN^r\\
&\le \, e(\HH) p^k\, +\, Crk^r c t^k N^r 
\end{align*}
where we assume in the last line that our choice of $c$ will be at most $1$.  The exact form of the argument now depends on the value of $\mu=p^ke(\HH)$.  

If $\mu\ge \eps t^kN^r/4$ then we have 
\[
\sum_{e\in E(\HH)}\pr{e\subseteq B_p|U\subseteq B_p}\, \le \, \mu(1+\eps/2)
\]
provided $c <\eps^2/8Ck^{k+1}$.

If $\mu\le \eps t^k N^r/4$ then we have
\[
\sum_{e\in E(\HH)}\pr{e\subseteq B_p|U\subseteq B_p}\, \le \,\frac{\eps t^k N^r}{2}
\]
provided $c <\eps/4Ck^{k+1}$.  These bounds give us ~\eqr{begin}.

The lemma now follows by a straightforward induction argument, as 
\begin{align*}\Ex{N^{\HH}(B_p)^{\ell}}\, &=\, \sum_{e_1,\dots ,e_\ell\in E(\HH)}\pr{e_1\cup \dots \cup e_{\ell}\subseteq B_p}\\
&= \, \sum_{e_1,\dots ,e_{\ell-1}\in E(\HH)} \pr{e_1\cup \dots \cup e_{\ell-1}\subseteq B_p}\sum_{e_{\ell}\in e(\HH)}\pr{e_{\ell}\subseteq B_p|e_1\cup\dots \cup e_{\ell-1}\subseteq B_p}\\
& \le\, \max\left\{\mu(1+\eps/2)\, , \, \frac{\eps t^kN^r}{2}\right\}\Ex{N^{\HH}(B_p)^{\ell-1}}\, .
\end{align*}
Note that the last inequality was obtained using our bound~\eqr{begin}.
\end{proof}

We now deduce Proposition~\ref{prop:known}.

\begin{proof}[Proof of Proposition~\ref{prop:known}]  Let us set $s:=i/N$.  Note that $s\le t=m/N$.  Throughout the proof we set $\mu=\Ex{N^{\HH}(B_s)}=s^ke(\HH)$.

It may be easily verified (using Stirling's approximation for example) that $\pr{\Bin(N,s)=i}\ge 1/N$, and so it suffices to prove the bounds
\eq{mustl}
\pr{N^{\HH}(B_s)\, \le\, \mu\, -\, \eps t^k N^r}\, \le\, \exp(-ct^{k/r}N)\, 
\eqe
and
\eq{mustu}
\pr{N^{\HH}(B_s)\,\ge \, \mu\, +\, \eps t^k N^r}\, \le\, \exp(-ct^{k/r}N)\, 
\eqe
for some $c=c(k,C,\eps)>0$.  We have used here that the difference between the means in the $B_i$ and $B_s$ models is $O(t^{k-1}N^{r-1})$ and so is negligible.

We begin by proving~\eqr{mustl}, the bound on the lower tail.  We do so using Janson's inequality, see Theorem 2.14 of~\cite{JLR}.  We note that the bound on $\Delta_r$ gives us that $\Delta_j\le CN^{r-j}$ for all $1\le j\le r$ and $\Delta_j\le C$ for all $j\ge r$.  We shall use these bounds to control the quantity $\overline{\Delta}$ which occurs in Janson's inequality.  We have
\begin{align*}
\overline{\Delta}\, &=\, \sum_{e,f\in E(\HH)} \pr{e\cup f\subseteq B_s}1_{e\cap f\neq \emptyset}\\
& \le \, \sum_{j=1}^{r} \binom{N}{j} (CN^{r-j})^2 s^{2k-j}\, +\, e(\HH)2^k C s^k\\
& \le \, C^2 \sum_{j=1}^{r} N^{2r-j}s^{2k-j}\, +\, 2^k C^2 s^kN^r\\
& \le\,  rC^2 t^{2k-1}N^{2r-1}\, +\, 2^k C^2 t^kN^r\\
& \le \, C' \max\{t^{2k-1}N^{2r-1},t^kN^r\}
\end{align*}
for some constant $C'=C'(C,k)$.  We may now apply Janson's inequality to obtain for some $c=c(k,C,\eps)>0$ that
\begin{align*}
\pr{N^{\HH}(B_s)\, \le\, s^ke(\HH)\, -\, \eps t^k N^r}\, &\le\, \exp\left(\frac{-\eps^2 t^{2k}N^{2r}}{2\overline{\Delta}}\right)\\
&\le\, \exp\left(\frac{-\eps^2 t^{2k}N^{2r}}{2C' \max\{t^{2k-1}N^{2r-1},t^kN^r\}}\right)\\
&\le \, \exp(-c\min\{tN,t^{k}N^{r}\})\\
&\le\, \exp(-ct^{k/r}N)\, ,
\end{align*}
where for the last inequality we used that $k\ge r$ and $t\ge N^{-r/k}$.

For the upper tail,~\eqr{mustu}, we use Lemma~\ref{lem:moment}, which (applied with $\eps/C$) gives us that
\[
\Ex{N^{\HH}(B_s)^{\ell}}\, \le\, \max\left\{\mu^{\ell}(1+\eps/2C)^\ell\, , \, \left(\frac{\eps t^kN^r}{2C}\right)^\ell\right\}\, ,
\]
where $\ell= c't^{k/r}N$ for some constant $c'>0$.  We now simply apply Markov.  We consider two cases based on which term is larger.  We note that $\mu\, s^ke(\HH)\, \le \, t^kN^r\Delta_r\, \le\, Ct^kN^r$.

If the maximum is the first term then by Markov we have
\begin{align*}
\pr{N^{\HH}(B_s)\,\ge \, \mu\, +\, \eps t^k N^r}\,&  \le\, \pr{N^{\HH}(B_s)\,\ge \, \left(1+\frac{\eps}{C}\right)\mu}\\
& =\, \pr{N^{\HH}(B_s)^{\ell}\,\ge \, \left(1+\frac{\eps}{C}\right)^{\ell}\mu^{\ell}}\\
& \le\, \frac{(1+\eps/2C)^{\ell}}{(1+\eps/C)^{\ell}}\, ,
\end{align*}           
which is of the form $\exp(c\ell)=\exp(-c t^{k/r}N)$, as required.

If the maximum is the second term then 
\begin{align*}
\pr{N^{\HH}(B_s)\,\ge \, \mu\, +\, \eps t^k N^r}\,& \le\, \pr{N^{\HH}(B_s)\,\ge \, \eps t^k N^r}\\
& \le \, \frac{(\eps t^k N^r/2C)^{\ell}}{(\eps t^k N^r)^{\ell}}\\
&=\, (2C)^{-\ell}\\
&\le \, \exp(-ct^{k/r}N)\, ,
\end{align*} 
for some constant $c>0$, as required.
\end{proof}

We now state the application which will be of interest to us.  In fact we require a result which is not only for link hypergraphs but also some hypergraphs derived from them.  Given a $k$-uniform hypergraph $\HH$ on vertex set $[N]$ and a vertex $x\in [N]$ the \emph{link hypergraph} $\HH(x)$ is the $(k-1)$-uniform hypergraph with vertex set $V\setminus \{x\}$ and an edge $e\setminus \{x\}$ for each edge $e\in E(\HH)$.  In the case of a weighted hypergraph $e\setminus \{x\}$ inherits the weight of $e$.

Let us also consider the operation in which each edge is replaced by all of its $j$-element subsets.  Given a $k$-uniform hypergraph $\HH$ and $j\le k$ we write $\HH_j$ for the (weighted) $j$-uniform hypergraph in which each edge is replaced by its $j$-element subsets (with multiplicity).  That is, the edges of $\HH_j$ are the $j$-element subsets which are contained in at least one edge of $\HH$, and the weight associated with an edge $f$ is $|\{e\in E(\HH): f\subseteq e\}|$, the number of edges of $\HH$ which contain it.

We will consider applying Proposition~\ref{prop:known} to link hypergraphs $\HH(x)$ and the hypergraphs $\HH(x)_j$ obtained from them.  Given an element $x\in [N]$ we write $B^{(x)}_{i-1}$ for a uniformly random set of $i-1$ elements of $V(\HH(x))=V(\HH(x)_j)=[N]\setminus \{x\}$.

\begin{corollary}\label{cor:link} Let $2\le r\le k$ and $C$ be integers and let $\eps>0$, then there is a constant $c>0$ such that the following holds.  Let $\HH$ be a $k$-uniform hypergraph with $\Delta_r(\HH)\le C$.  Let $t\ge N^{-(r-1)/(k-1)}$, and let $m=tN$.  
%
There is probability at most $2kN^3\exp(-ct^{(k-1)/(r-1)}N)$ that the inequality
\[
\max_{x\in [N]\setminus B_{i-1}}\, \big|D^{\HH(x)_j}(B_{i-1})\big|\, >\, \eps t^{j} N^{r-1}
\]
occurs for some $i\le m$ and some $1\le j\le k-1$.
\end{corollary}

\begin{proof} It clearly suffices to prove that for fixed choices of $i\le m$, $j\le k-1$ and $x\in [N]$ we have that
\[
\pr{x\not\in B_{i-1}\quad \text{and}\quad |D^{\HH(x)_j}(B_{i-1})|\, \ge\, \eps t^{j} N^{r-1}}\, \le\, 2N \exp(-ct^{(k-1)/(r-1)}N)\, .
\]
Let us fix $i\le m,j\le k-1$ and $x\in [N]$.  Since the event requires $x\not\in B_{i-1}$, it suffices to prove
\eq{inlink}
\pr{|D^{\HH(x)_j}(B^{(x)}_{i-1})|\, \ge\, \eps t^{j} N^{r-1}}\, \le \, 2N \exp(-ct^{(k-1)/(r-1)}N)\, 
\eqe
where $B^{(x)}_{i-1}$ is a uniformly random set of $i-1$ elements of $[N]\setminus \{x\}$.

During the proof we will use Proposition~\ref{prop:known}.  To avoid confusion with the parameters $r,k,C,\eps$ of the corollary we shall write $r',k',C'$ and $\eps'$ for the parameters in the condition of the proposition and $c=c(k',C',\eps')$ for the constant which appears in the result.  We consider two cases depending on the value of $j$.  Let us first observe that $\HH(x)$ inherits certain properties from $\HH$.  In particular $\HH(x)$ is $(k-1)$-uniform and has $\Delta_{r-1}\le C$.  In the case of $\HH(x)_j$ we note that $\HH(x)_j$ is $j$-uniform and we have the bounds $\Delta_{r-1}\le C2^k$ ($j\ge r-1$) and $\Delta_j\le C2^kN^{r-j-1}$ ($j\le r-2$).

Case I: If $r-1\le j\le k-1$ then we apply Proposition~\ref{prop:known} with parameters $r'=r-1$, $k'=j$, $C'=C2^k$ and $\eps'=\eps$ to the hypergraph $\HH(x)_j$ to obtain
\[
\pr{|D^{\HH(x)_j}(B^{(x)}_{i-1})|\, \ge\, \eps t^{j} N^{r-1}}\, \le \, 2N \exp(-ct^{j/(r-1)}N)\, \le  \, 2N \exp(-ct^{(k-1)/(r-1)}N)\, ,
\]
as required.

Case II: If $1\le j\le r-2$ then we let $\HH'$ be the weighted hypergraph obtained from $\HH(x)_j$ by dividing all its edge weights by $N^{r-j-1}$.  It then follows that $\HH$ is a $j$-uniform weighted hypergraph with $\Delta_j\le C2^k$.  We apply Proposition~\ref{prop:known} with parameter $r'=j$, $k'=j$, $C'=C2^k$ and $\eps'=\eps$ to the hypergraph $\HH'$ to obtain
\begin{align*}
\pr{|D^{\HH(x)_j}(B^{(x)}_{i-1})|\, \ge\, \eps t^{j} N^{r-1}}\, &= \, \pr{|D^{\HH'}(B^{(x)}_{i-1})|\, \ge\, \eps t^{j} N^{j}}\\
&\le \, 2N \exp(-ctN)\\
&\le \, 2N\exp(-ct^{(k-1)/(r-1)}N)\, ,
\end{align*}
as required.
\end{proof}

\section{Understanding $X_{\ell}(B_i)$ in terms of $X_{1}(B_i)$}\label{sec:ellto1}

In this section we shall see that all the contributions $X_{\ell}(B_i)$ to the martingale increment are (with very high probability) close to a deterministic multiple of $X_1(B_i)$.  Set
\[
Y_{\ell}(B_i)\, :=\, X_{\ell}(B_i)\, -\, \binom{k-1}{\ell - 1} \frac{(i-1)_{\ell-1}}{(N-1)_{\ell-1}}X_{1}(B_i)\, .
\]
The main result of the section is as follows.  We denote by $Y_{\ell}(B_{i-1},\cdot)$ the random variable $Y_{\ell}(B_i)$ conditioned on the first $i-1$ elements.  So that, for example, 
\[
\|Y_{\ell}(B_{i-1},\cdot)\|_{\infty}\, :=\, \max_{x\in [N]\setminus B_{i-1}}\big|Y_{\ell}(B_{i-1}\cup \{x\})\big|\, ,
\]
which is a $B_{i-1}$-measurable quantity.

\begin{prop}\label{prop:Ysmall} Let $2\le r\le k$ and $C$ be integers and let $\eps>0$, then there is a constant $c>0$ such that the following holds.  Let $\HH$ be a $k$-uniform hypergraph with $\Delta_r(\HH)\le C$.  Let $t\ge N^{-(r-1)/(k-1)}$, and let $m=tN$.  Except with probability at most $2kN^3\exp(-ct^{(k-1)/(r-1)}N)$ we have
\[
\|Y_{\ell}(B_{i-1},\cdot)\|_{\infty} \, \le\, \eps t^{\ell-1} N^{r-1}
\]
for all $i\le m$ and all $\ell\le k$.
\end{prop}

The alert reader will notice the similarity with our bound, Corollary~\ref{cor:link}, on deviations in link hypergraphs $\HH(x)$ and the related hypergraphs $\HH(x)_j$.  This is no coincidence.  Let $E(\eps)$ be the event that 
\[
\max_{x\in [N]\setminus B_{i-1}}\, \big|D^{\HH(x)_j}(B_{i-1})\big|\, \le\, \eps t^{j} N^{r-1}/2
\]
for all $i\le m$ and $1\le j\le k-1$.  By Corollary~\ref{cor:link} we have $\pr{E(\eps)}\ge 1-2kN^3\exp(-ct^{(k-1)/(r-1)}N)$.  And so it suffices to prove that 
\eq{stpy}
 \|Y_{\ell}(B_{i-1},\cdot)\|_{\infty}\, \le\, 2\max_{x\in [N]\setminus B_{i-1}}|D^{\HH(x)_{\ell-1}}(B_{i-1})|\qquad a.s.
\eqe
This is the approach we will take to proving Proposition~\ref{prop:Ysmall}.

\begin{proof}[Proof of Proposition~\ref{prop:Ysmall}] As discussed above it suffices to prove~\eqr{stpy}.  Let us recall the definition of $Y_{\ell}(B_i)$ and expand.  We obtain
\begin{align*}
Y_{\ell}(B_i)\,& =\, X_{\ell}(B_i)\, -\, \binom{k-1}{\ell - 1} \frac{(i-1)_{\ell-1}}{(N-1)_{\ell-1}}X_{1}(B_i)\phantom{\bigg|}\\
&=\, A_{\ell}(B_i)\, -\, \Ex{A_{\ell}(B_i)|B_{i-1}}\phantom{\bigg|}\\
&\qquad  -\, \binom{k-1}{\ell - 1} \frac{(i-1)_{\ell-1}}{(N-1)_{\ell-1}}A_1(B_i)\, +\, \binom{k-1}{\ell - 1} \frac{(i-1)_{\ell-1}}{(N-1)_{\ell-1}}\Ex{A_{1}(B_i)|B_{i-1}} \phantom{\bigg|}\\
&=\, T_{\ell}(B_i)\, -\, \Ex{T_{\ell}(B_i)|B_{i-1}}\phantom{\bigg|}
\end{align*}
where 
\[
T_{\ell}(B_i) \, :=\, A_{\ell}(B_i)\, -\, \binom{k-1}{\ell - 1} \frac{(i-1)_{\ell-1}}{(N-1)_{\ell-1}}A_{1}(B_i)\, .
\]
Using the notation $T_{\ell}(B_{i-1},\cdot)$ for the random variable $T_{\ell}(B_i)$ conditioned on the first $i-1$ elements, it clearly suffices to prove that
\[
\|T_{\ell}(B_{i-1},\cdot)\|_{\infty}\, \le\, \max_{x\in [N]\setminus B_{i-1}}|D^{\HH(x)_{\ell-1}}(B_{i-1})|\, \qquad a.s.
\]
In fact we shall prove that $T_{\ell}(B_{i-1}\cup \{x\})=D^{\HH(x)_{\ell-1}}(B_{i-1})$ for all $x\in [N]\setminus B_{i-1}$.  We observe that 
\begin{align*}
D^{\HH(x)_{\ell-1}}(B_{i-1})\,  &=\, N^{\HH(x)_{\ell-1}}(B_{i-1})\, -\, \Ex{N^{\HH(x)_{\ell-1}}(B^{(x)}_{i-1})}\phantom{\Bigg|}\\
&=\,  N^{\HH(x)_{\ell-1}}(B_{i-1})\, -\, \frac{(i-1)_{\ell-1}}{(N-1)_{\ell-1}} \binom{k-1}{\ell - 1}e(\HH(x))\phantom{\Bigg|}\\
&=\,N^{\HH(x)_{\ell-1}}(B_{i-1})\, -\, \binom{k-1}{\ell - 1} \frac{(i-1)_{\ell-1}}{(N-1)_{\ell-1}}A_{1}(B_{i-1}\cup \{x\})\, .\phantom{\Bigg|}
\end{align*}
And so it remains only to prove that $A_{\ell}(B_{i-1}\cup \{x\})= N^{\HH(x)_{\ell-1}}(B_{i-1})$.  Observe that $A_{\ell}(B_{i-1}\cup\{x\})$ counts the number of pairs $(S,e)$ with
\begin{enumerate}
\item[(i)] $S$ is an $\ell$-element subset of $[N]$ with $x\in S$
\item[(ii)]  $e\in E(\HH)$ 
\item[(iii)] $S\subseteq e$
\item[(iv)] $S\setminus \{x\}\subseteq B_{i-1}$.
\end{enumerate}
The conditions (i)-(iii) correspond to $(\ell-1)$-sets $S'=S\setminus \{x\}$ in edges $e\setminus \{x\}$ of the link hypergraph $\HH(x)$.  In this context, condition (iv) asserts that $S'\subseteq  B_{i-1}$.  It follows that $A_{\ell}(B_{i-1}\cup\{x\})\, =\, N^{\HH(x)_{\ell-1}}(B_{i-1})$, as required.
\end{proof}

\section{Approximating $D^{\HH}(B_m)$ by $\Lambda^{\HH}(B_m)$}\label{sec:approx}

As we discussed in the introduction, a major task on our path to proving Theorem~\ref{thm:main} involves approximating $D^{\HH}(B_m)$ by $\Lambda^{\HH}(B_m)$, which depends only on the degrees of the vertices selected for the set $B_m$.  We may now state explicitly the sense in which $\Lambda^{\HH}(B_m)$ approximates $D^{\HH}(B_m)$.  We use the notation $\omega(1)$ for a function which tends to $\infty$ as $N\to \infty$.

\begin{prop}\label{prop:approx} 
Let $2\le r\le k$ and $C$ be integers.  Let $\HH_N$ be a sequence of (weighted) $k$-uniform hypergraphs with $V(\HH_N) = [N]$, $\Delta_r(\HH_N)\le C$.  Let $t=m/N$ satisfy $(\log{N}/N)^{(r-1)/(k-1)}\ll t \le 1/2$ and let $\alpha_N$ be a sequence such that $\alpha_N\, \ll\, t^{k-1/2+(k-1)/2(r-1)}N^r$ \, .
Then
\[
\pr{\big|D^{\HH_N}(B_m)-\Lambda^{\HH_N}(B_m)\big|\,  \ge\, \alpha_N}\, \le \, \exp \left(  \frac{-\omega(1)\alpha_N^2}{t^{2k-1}N^{2r-1}} \right)\, .
\]
\end{prop}

To lighten the notation we drop $N$ from the notation $\HH_N$ in this section.  

Let us now give an idea of the proof of Proposition~\ref{prop:approx}.  We recall that we have a martingale expression for each of 
\[
D^{\HH}(B_m)\, =\,  \sum_{i=1}^{m} \sum_{\ell=1}^{k} \frac{(N-m)_{\ell}(m-i)_{k-\ell}}{(N-i)_k}\, \,  X_{\ell}^{\HH}(B_i)
\]
and we may express
\[
\Lambda^{\HH}(B_m)\, =\, \sum _{i=1}^{m} \kappa'(i,m)\, \,X^{\HH}_1(B_i)\, ,
\]
where $\kappa'(i,m):=\frac{t^{k-1}(1-t)}{1-s}$.  We will now express $D^{\HH}(B_m)$ in a similar form.  We recall the random variables
\[
Y_{\ell}(B_i)\, :=\, X_{\ell}(B_i)\, -\, \binom{k-1}{\ell - 1} \frac{(i-1)_{\ell-1}}{(N-1)_{\ell-1}}\, \, X_{1}(B_i)
\]
defined in the previous section and note that
\begin{align*}
D^{\HH}(B_m)\, & =\, \sum _{i=1}^{m} \kappa(i,m)\, \,X^{\HH}_1(B_i)\\
& +\, \sum_{i=1}^{m} \sum_{\ell=1}^{k}\frac{(N-m)_{\ell}(m-i)_{k-\ell}}{(N-i)_k}Y_{\ell}(B_i)\, 
\end{align*}
where 
\[
\kappa(i,m)\, :=\, \sum_{\ell=1}^{k}\frac{(N-m)_{\ell}(m-i)_{k-\ell}}{(N-i)_k}\binom{k-1}{\ell - 1} \frac{(i-1)_{\ell-1}}{(N-1)_{\ell-1}}\, .
\]
And so the difference between $D^{\HH}(B_m)$ and $\Lambda^{\HH}(B_m)$ may be expressed as:
\begin{align}
D^{\HH}(B_m)\, -\, \Lambda^{\HH}(B_m)\, &=\, \sum _{i=1}^{m} \big(\kappa(i,m)-\kappa'(i,m) \big)\, \,X^{\HH}_1(B_i)\label{eq:diff}\\
& +\, \sum_{i=1}^{m} \sum_{\ell=1}^{k}\frac{(N-m)_{\ell}(m-i)_{k-\ell}}{(N-i)_k}Y_{\ell}(B_i)\, .\nonumber
\end{align}

The following lemma will be useful.

\begin{lemma}\label{lem:kappadiff}
For all $i\le m\le N/2$ we have 
\[
\big|\kappa(i,m)\, -\, \kappa'(i,m)|\, \le\, O(t^{k-2}N^{-1})\, .
\]
\end{lemma}

\begin{proof} It is easily observed (binomial theorem) that
\[
\kappa'(i,m)\, =\, \sum_{\ell=1}^{k}\frac{(1-t)^{\ell}(t-s)^{k-\ell}}{(1-s)^k}\binom{k-1}{\ell - 1} s^{\ell-1}\, .
\]
So that it suffices to prove that
\[
\left|\frac{(N-m)_{\ell}(m-i)_{k-\ell}}{(N-i)_k}\frac{(i-1)_{\ell-1}}{(N-1)_{\ell-1}}\, -\, \frac{(1-t)^{\ell}(t-s)^{k-\ell}}{(1-s)^k}s^{\ell-1}\right|\, =\, O(t^{k-2}N^{-1})\, .
\]
This is now straightforward since the terms in the first expression are $(N-m)_{\ell}=(1-t+O(N^{-1}))^{\ell}N^{\ell}$, $(m-i)_{k-\ell}=(t-s+O(N^{-1}))^{k-\ell}N^{k-\ell}$, $(N-i)_k=(1-s+O(N^{-1}))^{k}N^k$, $(i-1)_{\ell-1}=(1+O(N^{-1}))s^{\ell-1}N^{\ell-1}$ and $(N-1)_{\ell-1}=(1+O(N^{-1}))N^{\ell-1}$.  When we multiply out, the main term is $(1-t)^{\ell}(t-s)^{k-\ell}s^{\ell-1}/(1-s)^k$ and all remaining terms are $O(t^{k-2}N^{-1})$, as required.
\end{proof}

We now prove Proposition~\ref{prop:approx}.

\begin{proof}[Proof of Proposition~\ref{prop:approx}] Let $S^{(1)}(m)$ and $S^{(2)}(m)$ denote the two terms on the right hand side of~\eqr{diff}.  By the triangle inequality, the union bound, and the expression~\eqr{diff} for the difference $D^{\HH}(B_m)-\Lambda^{\HH}(B_m)$ it suffices to prove
\[
\pr{S^{(i)}(m)\,  \ge\, \alpha_N}\, \le \, \exp \left(  \frac{-\omega(1)\alpha_N^2}{t^{2k-1} N^{2r-1}} \right)\,
\]
for $i=1,2$ and all $t, \alpha_N$ satisfying the conditions.

We begin with $S^{(1)}(m)$.  Setting $\kappa^{\Delta}(i,m)=\kappa(i,m)-\kappa'(i,m)$ we have that $S^{(1)}(m)=\sum_{i=1}^{m}\kappa^{\Delta}(i,m)\, X^{\HH}_1(B_i)$ and, by Lemma~\ref{lem:kappadiff}, $\kappa^{\Delta}=O(t^{k-2}N^{-1})$.  This will allow us to bound the maximum change and the quadratic variation of the martingale $S^{(1)}(j):j=0,\dots ,m$ defined by $S^{(1)}(j)=\sum_{i=1}^{j}\kappa^{\Delta}(i,m)\, X^{\HH}_1(B_i)$.  In particular, as $|X^{\HH}_1(B_i)|\le \Delta_1(\HH)\le CN^{r-1}$ the increments of the martingale have absolute value at most $C_1t^{k-2} N^{r-2}$ almost surely.
By the Hoeffding-Azuma inequality (Lemma~\ref{lem:HA}) we obtain
\begin{align*}
\pr{S^{(1)}\,  \ge\, \alpha_N}\, & \le \, \exp \left(  \frac{-\alpha_N^2}{2m C_1^2 t^{2k-4}N^{2r-4}} \right)\\
&=\, \exp \left(  \frac{-\alpha_N^2}{2C_1^2 t^{2k-3}N^{2r-3}} \right)\\
&\le \, \exp \left(  \frac{-\omega(1)\alpha_N^2}{t^{2k-1} N^{2r-1}} \right)\, ,
\end{align*}
where we have used here that $t\gg (\log{N}/N)^{(r-1)/(k-1)}\gg N^{-1}$.

Let $\eps>0$.  We now consider
\[
S^{(2)}(m)\, :=\, \sum_{i=1}^{m} \sum_{\ell=1}^{k}\frac{(N-m)_{\ell}(m-i)_{k-\ell}}{(N-i)_k}Y_{\ell}(B_i)\, .
\]
It will be convenient to work with
\[
Y^{*}_{\ell}(B_i)\, :=\, Y_{\ell}(B_i)1_{\|Y_{\ell}(B_{i-1},\cdot)\|_{\infty} \, \le\, \eps t^{\ell-1} N^{r-1} }\, .
\]
This ensures that $\|Y^{*}_{\ell}(B_i)\|_{\infty}\, \le\, \eps t^{\ell-1} N^{r-1}$ almost surely.  Also, by Proposition~\ref{prop:Ysmall}, we have that $Y^{*}_{\ell}(B_i)=Y_{\ell}(B_i)$ for all $i\le m$ and $\ell\le k$, except with probability at most $O(N^3)\exp(-\Omega(t^{(k-1)/(r-1)}N))$.  It follows that 
\[
S^{(2),*}(m)\, :=\, \sum_{i=1}^{m} \sum_{\ell=1}^{k}\frac{(N-m)_{\ell}(m-i)_{k-\ell}}{(N-i)_k}Y^{*}_{\ell}(B_i)\, 
\]
is equal to $S^{(2)}(m)$ except with probability at most $O(N^3)\exp(-\Omega(t^{(k-1)/(r-1)}N))$.  Since the coefficient of $Y^{*}_{\ell}(B_i)$ is at most $t^{k-\ell}$ it is easily checked that the increments of the martingale $S^{(2),*}(m)$ are all at most $\eps kt^{k-1}N^{r-1}$ almost surely.
We now apply the Hoeffding-Azuma inequality to obtain
\begin{align*}
\pr{S^{(2)}(m)\,  \ge\, \alpha_N}\, & \le \, \pr{S^{(2),*}(m)\,  \ge\, \alpha_N}\, +\, O(N^3)\exp(-\Omega(t^{(k-1)/(r-1)}N))\\
& \le \, \exp\left(\frac{-\alpha_N^2}{2m(\eps k t^{k-1}N^{r-1})^2}\right)\, +\, O(N^3)\exp(-\Omega(t^{(k-1)/(r-1)}N))\\
& \le \, \exp\left(\frac{-\alpha_N^2}{2\eps^2 k^2t^{2k-1}N^{2r-1}}\right)\, +\, O(N^3)\exp(-\Omega(t^{(k-1)/(r-1)}N))\, .
\end{align*}
Since $\eps$ is arbitrary, and the second probability is much smaller (by the lower bound on $t$ and the upper bound on $\alpha_N$), this completes the proof.
\end{proof}

\section{The quadratic variation of the process}\label{sec:var}

Given a hypergraph $\HH$ we may write $X^{\HH}(B_i)$ (or simply $X(B_i)$) for the martingale increment
\[
X^{\HH}(B_i) \, :=\,  \frac{t^{k-1}(1-t)}{1-s}\, X^{\HH}_1(B_i)
\]
of the $\Lambda^{\HH}(B_i)$ process.

Since our eventual aim is to control deviation probabilities using Freedman's inequality, the behaviour of the quadratic variation
\[
V^{\HH}(m)\, :=\, \sum_{i=1}^{m}\left(\frac{t^{k-1}(1-t)}{1-s}\right)^2 \, \Ex{X^{\HH}(B_i)^2|B_{i-1}}
\]
of the process $\Lambda^{\HH}(B_m)$ is of particular importance.  We prove the following.

\begin{prop}\label{prop:var}
Let $2\le r\le k$ and $C$ be integers and let $\eps>0$, then there is a constant $c>0$ such that the following holds.  Let $\HH$ be a $k$-uniform hypergraph with $\Delta_r(\HH)\le C$.  Let $t\le 1/2$, and let $m=tN$.  Then, except with probability at most $4N^2\exp(-ctN)$, we have
\[
\left|V^{\HH}(m)\, -\, t^{2k-1}(1-t)\sigma^2(\HH) N\right| \, \le\, \eps t^{2k-1}N^{2r-1}\, .
\]
\end{prop}

The proof of Proposition~\ref{prop:var} will be relatively straightforward once we have proved the following lemma on the likely behaviour of $\Ex{X^{\HH}_1(B_i)^2|B_{i-1}}$.  To streamline notation we drop $\HH$ from the notation for the remainder of the section, writing simply, $V(m),X(B_i),X_1(B_i)$, etc.

\begin{lemma}\label{lem:var}
Let $2\le r\le k$ and $C$ be integers and let $\eps>0$, then there is a constant $c>0$ such that the following holds.  Let $\HH$ be a $k$-uniform hypergraph with $\Delta_r(\HH)\le C$.  Let $t\le 1/2$, and let $m=tN$.  Except with probability at most $4N^2\exp(-ctN)$ we have
\[
\big|\Ex{X_1(B_i)^2|B_{i-1}}\, -\, \sigma^2(\HH)\big| \, \le\, \eps N^{2r-2}
\]
for all $i\le m$.
\end{lemma}

\begin{proof} We begin by expanding $X_1(B_i)$ according to its definition in terms of $A_1(B_i)$, which is the degree of the vertex added at step $i$,
\begin{align*} 
\Ex{X_1(B_i)^2|B_{i-1}}\,& =\,\Ex{\big(A_1(B_i)-\Ex{A_1(B_i)|B_{i-1}}\big)^2|B_{i-1}}\\
& =\, \Ex{A_1(B_i)^2|B_{i-1}}\, -\, \big(\Ex{A_1(B_i)|B_{i-1}}\big)^2\, .
\end{align*}
We now consider each of these two terms and show that each is very likely to be close to a certain value.

Let $\HH_{deg}$ be the $1$-uniform hypergraph on $[N]$ in which every vertex is an edge with weight given by the degree of that vertex in $\HH$.  And let $\HH'_{deg}$ be obtained by dividing all these weights by $N^{r-1}$.

We will relate the deviation of $\Ex{A_1(B_i)|B_{i-1}}$ from its mean, $\bar{d}(\HH)=N^{-1}\sum_x d(x)$, to a deviation for the hypergraph $\HH'_{deg}$.  We first observe that
\begin{align*}
\Ex{A_1(B_i)|B_{i-1}}\, -\, \bar{d}(\HH)\,& =\, \frac{1}{N-i+1}\sum_{x\in [N]\setminus B_{i-1}}d(x)\, -\, \bar{d}(\HH)\\
& =\, \frac{i-1}{N-i+1}\, \bar{d}(\HH)\, -\, \frac{1}{N-i+1}\sum_{x\in B_{i-1}}d(x)\\
& =\, \frac{1}{N-i+1}\left(\Ex{N^{\HH_{deg}}(B_{i-1})}\, -\, N^{\HH_{deg}}(B_{i-1})\right)\\
& =\, \frac{-1}{N-i+1}D^{\HH_{deg}}(B_{i-1})\, .
\end{align*}
Since $i\le m=tN\le N/2$ we have $N-i+1\ge N/2$ and so
\begin{align*}
|D^{\HH'_{deg}}(B_{i-1})|\, &=\, |D^{\HH_{deg}}(B_{i-1})|/N^{r-1}\\
&\ge\, \frac{|\Ex{A_1(B_i)|B_{i-1}}\, -\, \bar{d}(\HH)|}{2N^{r-2}}\, .
\end{align*}
We now control this deviation using Proposition~\ref{prop:known}.  We have that $\HH'_{deg}$ is a $1$-uniform hypergraph with $\Delta_1\le C$ and so there is a constant $c(C,\eps')>0$ such that
\begin{align*} 
\pr{\big|\Ex{A_1(B_i)|B_{i-1}}\, -\, \bar{d}(\HH)\big|\, \ge\, \eps' N^{r-1}}\, &\le\, \pr{|D^{\HH'_{deg}}(B_{i-1})|\, \ge\, \eps' N/2}\\
&\le\,  2N\exp(-ctN)\, .
\end{align*}

The same argument may be used to control the deviation of $\Ex{A_1(B_i)^2|B_{i-1}}$ from its mean
\[
\bar{d}^{(2)}(\HH)\, :=\, \frac{1}{N}\sum_{x}d(x)^2\, .
\]
Simply proceed as above except with the hypergraph $\HH_{deg,2}$ with weights given by $d(x)^2$ and a renormalised version $\HH'_{deg,2}$ obtained by dividing all weights by $N^{2r-2}$.  It follows that 
\begin{align*} ]
\pr{\big|\Ex{A_1(B_i)^2|B_{i-1}}\, -\, \bar{d}^{(2)}(\HH)\big|\, \ge\, \eps' N^{2r-2}}\, &\le\, \pr{|D^{\HH'_{deg,2}}(B_{i-1})|\, \ge\, \eps' N/2}\\
&\le\,  2N\exp(-ctN)\, ,
\end{align*}
for some constant $c(C^2,\eps')>0$.

Taking $\eps'=\eps/4C$ the result now follows from the triangle inequality.  In particular, if neither of the deviations above occur then we have
\begin{align*}
\Ex{X_1(B_i)^2|B_{i-1}}\, &=\, \Ex{A_1(B_i)^2|B_{i-1}}\, -\, \big(\Ex{A_1(B_i)|B_{i-1}}\big)^2\\
& =\, \bar{d}^{(2)}(\HH)\, \pm \, \eps'N^{2r-2}\, -\, \big(\bar{d}(\HH)\pm \eps'N^{r-1}\big)^2\\
& =\, \bar{d}^{(2)}(\HH)\, -\, \bar{d}(\HH)^2\,  \pm \, 2\eps'N^{2r-2}\, \pm \, 2\eps'N^{r-1}\bar{d}(\HH)\\
& =\, \sigma^2(\HH)\, \pm\, 2\eps'N^{2r-2}\, \pm \, 2C\eps'N^{2r-2}\\
&=\, \sigma^2(\HH)\, \pm\, \eps N^{2r-2}\, . \qedhere
\end{align*}
\end{proof}

We may now deduce Proposition~\ref{prop:var}.

\begin{proof}[Proof of Proposition~\ref{prop:var}] Let us begin by considering a sum related to the coefficients $t^{k-1}(1-t)/(1-s)$ which occur in the definition of $X(B_i)$.  In particular, let us study the sum of the squares of these coefficients.  We will use that for $(i-1)/N\le s'\le i/N$ we have $(1-s')^{-2}\le (1-s)^{-2}\le (1-s')^{-2}+O(N^{-1})$.  We have
\begin{align*}
\sum_{i=1}^{m} \left(\frac{t^{k-1}(1-t)}{1-s}\right)^2\, &=\, t^{2k-2}(1-t)^2 \sum_{i=1}^{m}\frac{1}{(1-s)^2}\\
&=\, (1+O(N^{-1})) t^{2k-2}(1-t)^2N\, \int_{0}^{t}(1-s')^{-2}\,\, ds'\\
&=\, (1+O(N^{-1})) t^{2k-1}(1-t)N\, .
\end{align*}

We are now ready to deduce the proposition from Lemma~\ref{lem:var}.  By the lemma there exists a constant $c>4\eps^{-1}$ such that, except with probability at most $4N^2\exp(-ctN)$, the conditional second moments satisfy
\[
\big|\Ex{X_1(B_i)^2|B_{i-1}}\, -\, \sigma^2(\HH)\big| \, \le\, \eps N^{2r-2}/4
\]
for all $i\le m$.  If this is the case then 
\begin{align*}
V(m)\, &=\, \sum_{i=1}^{m}\, \left(\frac{t^{k-1}(1-t)}{1-s}\right)^2  \big(\sigma^2(\HH)\, \pm \, \eps N^{2r-2}/4\big)\\
& =\, (1+O(N^{-1})) t^{2k-1}(1-t)\sigma^2(\HH)N\, \pm\, \eps t^{2k-1}(1-t)N^{2r-1}/2\\
& =\, t^{2k-1}(1-t)\sigma^2(\HH)N\, \pm\, \eps t^{2k-1}N^{2r-1}\, ,
\end{align*} 
as required.
\end{proof}

\section{Proof of Theorem~\ref{thm:main}}\label{sec:main}

As we mentioned in the introduction, the proof of Theorem~\ref{thm:main} divides naturally into the task of proving that $D^{\HH}(B_m)$ is generally very well approximated by $\Lambda^{\HH}(B_m)$, which we have now achieved in the form of Proposition~\ref{prop:approx}, and the task of controlling the probability of deviations for $\Lambda^{\HH}(B_m)$.  We now state the required result for $\Lambda^{\HH}(B_m)$.  It will then be relatively straightforward to complete the proof of Theorem~\ref{thm:main}.

\begin{prop}\label{prop:main} Let $2\le r\le k$ and $C$ be integers.  Let $\HH_N$ be a sequence of (weighted) $k$-uniform hypergraphs with $V(\HH_N) = [N]$, $\Delta_r(\HH_N)\le C$ and $\sigma^2(\HH_N)\ge N^{2r-2}/C$ for all $N$.  Let $log{N}/N\ll m/N =t \le 1/2$.  Let $a_N$ be a sequence such that
\[
t^{k-1/2}N^{r-1/2}\, \ll\, a_N\, \ll\, t^{k}N^r \, .
\]
Then
\[
\pr{\Lambda^{\HH_N}(B_m)\,  \ge\, a_N}\, = \, \exp \left(  \frac{-(1+o(1))a_N^2}{2(1-t)t^{2k-1}\sigma^2(\HH_N) N} \right)\, .
\]
Furthermore the same holds for the lower tail probability $\pr{\Lambda^{\HH_N}(B_m)\,  \le\, -a_N}$.
\end{prop}

\begin{proof} We remark that the statement on the lower tail follows from the same proof with the obvious minor adjustments.  We omit the subscript $N$ from $\HH_N$ during the proof.

The result consists of an upper bound and a lower bound on $\pr{\Lambda^{\HH}(B_m)\,  \ge\, a_N}$.  We begin with the upper bound, which will be proved using Freedman's inequality (Lemma~\ref{lem:Freedman}).  We consider the martingale given by
\[
S_p\, :=\,\sum_{i=1}^{p}\frac{t^{k-1}(1-t)}{1-s}\, X^{\HH}_1(B_i)\qquad p=0,\dots, m,
\]
which has initial value $0$ and final value $S_m=\Lambda^{\HH}(B_m)$.  Note that the quadratic variation of the process $V(m)$ is exactly $V^{\HH}(m)$ studied in Section~\ref{sec:var}.

Let $\eps>0$ and let $\beta=t^{2k-1}(1-t)\sigma^2(\HH)N+\eps t^{2k-1}N^{2r-1}$ and $R=Ct^{k-1}N^{r-1}$.  By Proposition~\ref{prop:var} we have that $V^{\HH}(m)$, the quadratic variation of the process, is at most $\beta$, except with probability at most $4N^2\exp(-ctN)$, for some constant $c>0$.  We may also observe that $|X^{\HH}_1(B_i)|\le \Delta(\HH) \le N^{r-1}\Delta_r(\HH)\le CN^{r-1}$ deterministically and so the increments are at most $Ct^{k-1}N^{r-1}$ deterministically.  And so it follows by an application of Freedman's inequality (Lemma~\ref{lem:Freedman}) that
\begin{align*}
\pr{\Lambda^{\HH}(B_m)\,  \ge\, a_N}\, &\le\, \pr{\Lambda^{\HH}(B_m)\,  \ge\, a_N \quad \text{and}\quad V(m)\le \beta}\, +\, \pr{V(m)>\beta}\\
&\le \, \exp\left(\frac{-a_N^2}{2(\beta+R a_N)}\right)\, +\, 4N^2\exp(-ctN)\\
&\le\, \exp\left(\frac{-a_N^2}{2t^{2k-1}(1-t)\sigma^2(\HH)N+2\eps t^{2k-1}N^{2r-1}+2Ra_N}\right)\, +\, 4N^2\exp(-ctN)\, .
\end{align*}
The upper bound on $a_N$ implies that $ctN\gg a_N^2/t^{2k-1}\sigma^2 N+\log{N}$ and $2Ra_N\le \eps t^{2k-1}N^{2r-1}$ for all sufficiently large $N$, and so, for all sufficiently large $N$ we have
\[
\pr{\Lambda^{\HH}(B_m)\,  \ge\, a_N}\, \le\, (1+o(1))\exp\left(\frac{-a_N^2}{2t^{2k-1}(1-t)\sigma^2(\HH)N\, (1+2C\eps)}\right)\, .
\]
Since $C$ is fixed and $\eps$ is arbitrary this gives the required upper bound
\[
\pr{\Lambda^{\HH}(B_m)\,  \ge\, a_N}\, \le\, \exp\left(\frac{-(1+o(1))a_N^2}{2t^{2k-1}(1-t)\sigma^2(\HH)N}\right)\, .
\]

We now prove the lower bound using the converse Freedman inequality (Lemma~\ref{lem:Freedman_converse}).    One slight problem that faces us is that the converse inequality gives a lower bound to the event that $S_p\ge \alpha$ for some $p\le m$ rather than for $p=m$.

Let $\eps>0$ and let $\beta=t^{2k-1}(1-t)\sigma^2(\HH)N-\eps t^{2k-1}N^{2r-1}$ and $R=Ct^{k-1}N^{r-1}$.  As above all increments are bounded by $R$ deterministically and by Proposition~\ref{prop:var} we have $V(m)\ge \beta$ except with probability at most $4N^2\exp(-ctN)$.  Let $\alpha=(1+\eps)a_N$.  We shall apply Lemma~\ref{lem:Freedman_converse} with this value of $\alpha$.  We obtain that 
\[
\pr{T_{\alpha}\le \beta}\, \ge\, \frac{1}{2} \exp \left( \frac{-\alpha^2(1+4\delta)}{2\beta}\right),
\]
where $\delta$ is minimal such that $\beta/\alpha \ge 9R\delta^{-2}$ and $\alpha^2/\beta \ge 16\delta^{-2}\log (64\delta^{-2})$.  The conditions on $a_N$ are such that we may take $\delta=\delta_N=o(1)$, and so we have that
\[
\pr{T_{\alpha}\le \beta}\, \ge\, \frac{1}{2} \exp \left( \frac{-(1+O(\eps))a_N^2}{2t^{2k-1}(1-t)\sigma^2(\HH)N}\right)\, .
\]
If $V(m)\ge \beta$ then this event implies that the hitting time $\tau_{\alpha}$ is at most $m$.  Since this event fails with probability at most $4N^2\exp(-ctN)$, which is $o(1)$ of the main term, we have
\[
\pr{\exists p\le m\,\,:\,\, S_p\ge \alpha}\, \ge\, \frac{1}{4} \exp \left( \frac{-(1+O(\eps))a_N^2}{2t^{2k-1}(1-t)\sigma^2(\HH)N}\right)\, 
\]
for all sufficiently large $N$.  To complete the proof we must subtract the probability of the event that we reach $\alpha=(1+\eps)a_N$ for some $S_p$, $p<m$ and then decrease so that $S_m<a_N$.  Conditional on the event that $S_p\ge \alpha$ for some $p<m$, the part of the martingale after the hitting time $\tau_{\alpha}$ is also a martingale with quadratic variation at most $O(t^{2k-1}(1-t)\sigma^2(\HH)N)$ except with probability at most $4N^2\exp(-ctN)$ (by Proposition~\ref{prop:var}) we obtain by Lemma~\ref{lem:Freedman} that the conditional probability of this event is at most
\[
\exp\left(\frac{-\Omega(1)a_N^2}{t^{2k-1}(1-t)\sigma^2(\HH)N}\right)\, +\, 4N^2\exp(-ctN)\, \le\, \frac{1}{2}\, 
\]
for all sufficiently large $N$.  It follows that 
\[
\pr{S_m\ge a_N}\, \ge\, \frac{1}{8}\exp\left( \frac{-(1+O(\eps))a_N^2}{2t^{2k-1}(1-t)\sigma^2(\HH)N}\right)
\] 
for all sufficiently large $N$.  Since $\eps$ is arbitrary this is the required result.
\end{proof}

We are now ready to combine the various auxiliary results and complete the proof of Theorem~\ref{thm:main}.

\begin{proof}[Proof of Theorem~\ref{thm:main}]  Again we remark that the statement on the lower tail follows from the same proof with the obvious minor adjustments.  

We must prove an upper bound and a lower bound on $\pr{D^{\HH}(B_m)\,  \ge\, a_N}$.  We begin with the upper bound.  Observe that the interval of value of $a_N$ is empty unless $t\gg (\log{N}/N)^{(r-1)/(k-1)}$ and so we may suppose that $t\gg (\log{N}/N)^{(r-1)/(k-1)}$ throughout.  Fix $\eps>0$.  By Propositions~\ref{prop:approx} and~\ref{prop:main} we have 
\begin{align*}
\pr{D^{\HH}(B_m)\, \ge\, a_N}\,& \le\, \pr{\Lambda^{\HH}(B_m)\,   \ge\, (1-\eps)a_N}\, +\, \pr{|D^{\HH}(B_m)\, -\, \Lambda^{\HH}(B_m)|\, \ge\, \eps a_N}\\
& \le\,  \exp \left(  \frac{-(1-\eps+o(1))^2 a_N^2}{2(1-t)t^{2k-1}\sigma^2(\HH) N} \right)\, +\, \exp \left(  \frac{-\omega(1)\eps^2 a_N^2}{t^{2k-1}N^{2r-1}}\right)\, .
\end{align*}
As $\sigma^2(\HH)$ is of order $N^{2r-2}$ it is clear that the second term is $o(1)$ of the first for all $\eps>0$.  As $\eps$ is arbitrary we obtain
\[
\pr{D^{\HH}(B_m)\,  \ge\, a_N}\, \le \, \exp \left(  \frac{-(1+o(1))a_N^2}{2(1-t)t^{2k-1}\sigma^2(\HH) N} \right)\, ,
\]
as required.

For the lower bound we again use Propositions~\ref{prop:approx} and~\ref{prop:main}.  We obtain
 \begin{align*}
\pr{D^{\HH}(B_m)\, \ge\, a_N}\,& \ge\, \pr{\Lambda^{\HH}(B_m)\,   \ge\, (1+\eps)a_N}\, -\, \pr{|D^{\HH}(B_m)\, -\, \Lambda^{\HH}(B_m)|\, \ge\, \eps a_N}\\
&\ge\,  \exp \left(  \frac{-(1+\eps+o(1))^2 a_N^2}{2(1-t)t^{2k-1}\sigma^2(\HH) N} \right)\, -\, \exp \left(  \frac{-\omega(1)\eps^2 a_N^2}{t^{2k-1}N^{2r-1}} \right)\, .
\end{align*}
Again the second term is $o(1)$ of the first for all $\eps>0$.  As $\eps$ is arbitrary we obtain
\[
\pr{D^{\HH}(B_m)\,  \ge\, a_N}\, \ge \, \exp \left(  \frac{-(1+o(1))a_N^2}{2(1-t)t^{2k-1}\sigma^2(\HH) N} \right)\, ,
\]
as required.
\end{proof}

\section{Deviations $D^{\HH}(B_p)$ -- Proof of Theorem~\ref{thm:pworld}}\label{sec:pworld}

In this section, we prove Theorem~\ref{thm:pworld}.  To streamline notation we simply write $\sigma$ for $\sigma(\HH_N)$, $\bar{d}$ for $\bar{d}(\HH_N)$ and $h$ for $e(\HH_N)$ throughout this section. We recall that if we condition that $B_p$ contains exactly $m$ elements then it is distributed as $B_m$, and so we have
\begin{equation}\label{eq:mtop}
\pr{D^{\HH_N}(B_p)\, \ge\, \delta_N p^k h}\, =\, \sum_{m=0}^{N}b_{N,p}(m)\, \pr{N^{\HH_N}(B_m)\, \ge\, (1+\delta_N)p^kh}\, ,
\end{equation}
where $b_{N,p}(m):= \pr{\Bin(N,p) = m}$.  Setting $q := 1 - p$, we use the following well known bound for the binomial distribution (for stronger results, see Theorem 1.13 in ~\cite{GGS2019}, which adapts an argument of Bahadur ~\cite{Bahadur}).
\eq{ourbin}
b_{N,p}(m)\, =\, \exp\left(-\frac{(1+o(1))x(m)^2}{2}\, +\, O(\log{N})\right),
\eqe
where $x(m)=(m-pN)/\sqrt{pqN}$, which holds provided $N^{-1}\le p\le 1/2$, and $\sqrt{pN}\ll m-pN\ll pN$.

It is useful to classify the possible choices of $m$ as follows.  For $\eta\in [0,1]$ let us define 
\[
m_{\eta}\, :=\, \left(1\, +\, \frac{\eta \delta_N}{k}\right)pN\, .
\]

We may think of \eqr{mtop} as offering us various ways to achieve the required deviation.  The term $m=m_0$ corresponds to a case where the number of points in $B_p$ is equal to its  expected value, $pN$, and all the work of achieving the deviation must be done in the $m$-model.  On the other hand $m=m_1$ corresponds to a large enough deviation in the number of points in $B_p$ that no (significant) deviation is required in the $m$-model, as $L^{\HH_N}(m_1)\approx (1+\delta_N)h$.  So we will be interested in the choice of $\eta\in [0,1]$ which minimises the total ``cost'' of the deviation.  In fact this will be achieved by
\[
\eta^*\, :=\, \frac{\bar{d}^2}{\bar{d}^2+\sigma^2}\, .
\]
Note that for $\eta\in [0,1]$ we have by a fairly simple computation
\[
L^{\HH_N}(m_{\eta})\, =\, p^k h\, +\, (1+o(1))\eta\delta_N p^k h\, +\, O(p^{k-1}h/N)\, .
\]
Therefore, achieving the deviation $N^{\HH_N}(B_{m_{\eta}})\ge (1+\delta_N)p^kh$ corresponds to
\eq{mdev}
D^{\HH_N}(B_{m_{\eta}})\, \ge\, (1-\eta+o(1))\delta_N p^k h \, ,
\eqe
since $p^{k-1}h/N = o(\delta_Np^kh)$.

Let us also define $\eta^{\circ}:=(\eta^{*})^{1/2}=\bar{d}/\sqrt{\bar{d}^2+\sigma^2}$, and $m^{\circ}:=m_{\eta^{\circ}}$.

We are now ready to prove Theorem~\ref{thm:pworld}.

\begin{proof}[Proof of Theorem~\ref{thm:pworld}]
Note that, under the given conditions, any additive error of order $O(\log{N})$ in the exponential may be included in the little $o$ term.  As this is equivalent to a multiplicative $N^{O(1)}$ term in front of the exponential it suffices to prove a result for the maximum contribution (as there are only $N$ values of $m$) in \eqr{mtop}.  In other words, we must prove that
\begin{equation}\label{eq:maxf(m)}
\max_{m} f(m) \, =\, \exp\left(\frac{-(1+o(1))\delta_N^2 p h^2}{2q(\bar{d}^2+\sigma^2)N}\,+\, O(\log{N})\right)\, ,
\end{equation}
where
\[
f(m) := b_{N,p}(m)\, \pr{N^{\HH_N}(B_m)\, \ge\, (1+\delta_N)p^kh}.
\]
We consider three regimes of $m$: (i) $m \le m_0$, (ii) $m_0 \le m \le m^{\circ}$ and (iii) $m \ge m^{\circ}$. 

Regime (i): Since $b_{N,p}(m) \le 1$ and as the event $N^{\HH_N}(B_m)\, \ge\, (1+\delta_N)p^kh$ is increasing in $m$, we have
\begin{align*}
f(m) &\le \pr{N^{\HH_N}(B_{m_0}))\, \ge\, (1+\delta_N)p^kh} \\
       &= \pr{D^{\HH_N}(B_{m_0})\, \ge\, (1+o(1))\delta_Np^kh},
\end{align*}
where we used~\eqr{mdev} in the last line. The conditions on $\delta_N$ allow us to apply Theorem~\ref{thm:main} and so we obtain
\begin{align*}
f(m) &\le \exp\left(\frac{-(1+o(1))\delta_N^2 p h^2}{2q\sigma^2N}\right)\\
&\le \exp\left(\frac{-(1+o(1))\delta_N^2 p h^2}{2q(\bar{d}^2+\sigma^2)N}+\, O(\log{N})\right).
\end{align*}
Regime (iii):  Since $\pr{N^{\HH_N}(B_m)\, \ge\, (1+\delta_N)p^kh} \le 1$ and $b_{N,p}(m) \le b_{N,p}(m^{\circ})$ for $m \ge m^{\circ}$, we obtain
\begin{align*}
f(m) &\le b_{N,p}(m^{\circ}) \\
&= \exp\left(-(1+o(1))\frac{(\eta^{\circ})^2\delta_N^2 p N}{2q k^2}+\, O(\log{N})\right)\, \\
&= \exp\left(\frac{-(1+o(1))\delta_N^2 p h^2}{2q(\bar{d}^2+\sigma^2)N}+\, O(\log{N})\right),
\end{align*}
where we have used the estimate~\eqr{ourbin} for $b_{N,p}(m)$ and the fact that $h=N\bar{d}/k$.

Regime (ii): We must take into account the contributions from both the binomial distribution and the deviations in the $m$-model. Since $m_0 \le m \le m^{\circ}$, it suffices to prove
\[
\max_{\eta\in [0,\eta^{\circ}]} f(m_{\eta})\, =\, \exp\left(\frac{-(1+o(1))\delta_N^2 p h^2}{2q(\bar{d}^2+\sigma^2)N}+\, O(\log{N})\right)\, .
\]
The corresponding value of $x$, $x(m_\eta)$, is
\[
x(m_{\eta})\, =\,\frac{\eta\delta_N p^{1/2}N^{1/2}}{kq^{1/2}}\, .
\]
Using~\eqr{ourbin}, it follows that
\[
b_N(m_{\eta})\, =\, \exp\left(-(1+o(1))\frac{\eta^2\delta_N^2 p N}{2q k^2}+\, O(\log{N})\right)\, .
\]
On the other hand, by~\eqr{mdev} and Theorem~\ref{thm:main}, we have
\begin{align*}
\pr{N^{\HH_N}(m_{\eta})\ge (1+\delta_N) h}\, &=\,  \pr{D^{\HH_N}(m_{\eta})\, \ge\, (1-\eta+o(1))\delta_N p^k h}\\
& = \, \exp\left(\frac{-(1+o(1))(1-\eta)^2\delta_N^2ph^2}{2q\sigma^2 N}\right)\\
& =\, \exp\left(\frac{-(1+o(1))(1-\eta)^2\delta_N^2p\bar{d}^2 N}{2qk^2 \sigma^2}\right)\, ,
\end{align*}
where we have used in the final line that $h=\bar{d}N/k$.  It follows that
\[
f(m_{\eta}) =\, \exp\left(\frac{-(1+o(1))\delta_N^2p N}{2qk^2 \sigma^2}\big[\sigma^2 \eta^2  \, +\, \bar{d}^2(1-\eta)^2\big]\, +\, O(\log{N})
\right)\, .
\]
This expression is maximized by choosing $\eta$ to minimize $\sigma^2 \eta^2+\bar{d}^2 (1-\eta)^2$.  This value of $\eta$ is 
\[
\eta^*\, =\, \frac{\bar{d}^2}{\bar{d}^2+\sigma^2}
\]
and in this case (using $\bar{d}=hk/N$) we have precisely
\[
f(m_{\eta^*})\, =\, \exp\left(\frac{-(1+o(1))\delta_N^2 p h^2}{2q(\bar{d}^2+\sigma^2)N}+\, O(\log{N})\right)\, ,
\]
as required.
\end{proof}

\section{Applications to arithmetic structures}\label{sec:appl}

In this section we specialize Theorems \ref{thm:main} and \ref{thm:pworld} to arithmetic progressions (Section \ref{sec:kap}) and additive quadruples (Section \ref{sec:sidon}) in $\{1, \ldots, N\}$.

\subsection{Arithmetic progressions}\label{sec:kap}

We denote by $\HH^{k}$ the hypergraph encoding increasing $k$-APs in $[N]$. Since applying Theorems \ref{thm:main} and \ref{thm:pworld} requires control of the average degree and of the variance of the degrees, the following lemma will be useful. Its proof is given in \cite{SeccoThesis} (we remark that similar computations were given in \cite{BGSZ}).

\begin{lemma}\label{lem:var_ap}
Consider the hypergraph $\HH^{k}$ of increasing $k$-APs in $[N]$. We have
\[
\bar{d}(\HH^{k}) = (1+o(1)) \frac{kN}{2(k-1)}
\]
and
\[
\sigma^2(\HH^k) = (1+o(1)) \theta_k N^2,
\]
where
\begin{equation}\label{eq:thetak}
\theta_k := \frac{1}{3(k-1)^2} \left( k - \frac{3k^2}{4} + \sum_{1\le i < j \le k} \frac{(k-1)^2 - (k-j)^2 - (i-1)^2}{(j-1)(k-i)} \right).
\end{equation}
\end{lemma}

It is also easy to check that $\Delta_2(\HH^k) = O(1)$ and so we can apply Theorems \ref{thm:main} and \ref{thm:pworld} with $r=2$. In the $m$-model, we obtain the following result.

\begin{theorem}\label{thm:kAPnonregm}
Let $0 \le m \le N$ be such that $t = m/N \le 1/2$. Let $a_N$ be a sequence such that
\[
t^{k-1/2}N^{3/2} (\log N)^{1/2} \ll a_N \ll t^{3k/2-1}N^2.
\]
Then
\[
\pr{D^{\HH^{k}}(B_m) \ge a_N} = \exp \left( \frac{-(1+o(1))a_N^2}{2\theta_k(1-t)t^{2k-1}N^3} \right),
\]
where $\theta_k$ is defined in \eqref{eq:thetak}.
\end{theorem}
 
Before stating the result for the $p$-model, we define
\[
\gamma_k = \frac{4}{3} \left( k + \sum_{1\le i < j \le k} \frac{(k-1)^2 - (k-j)^2 - (i-1)^2}{(j-1)(k-i)} \right)
\]
so that
\[
\frac{\bar{d}(\HH^k)^2 + \sigma^2(\HH^k)}{e(\HH^k)^2} = (1+o(1))\frac{\gamma_k}{N^2}.
\]
 
Therefore Theorem \ref{thm:pworld} gives us the following.
\begin{theorem}\label{thm:kAPnonregp}
Given a sequence $p = p_N$ which may be a constant in $(0,1/2)$ or converge to $0$, let $\delta_N$ be a sequence satisfying
\[
\sqrt{\frac{\log N}{pN}}  \ll \delta_N \ll p^{k/2-1}.
\]
Then
\[
\pr{ D^{\HH^{k}}(B_p)\, \ge\, \delta_N p^k e(\HH^k)} = \, \exp \left( \frac{-(1+o(1))\delta_N^2pN}{2\gamma_k(1-p)} \right).
\]
\end{theorem}
 
This result is applicable when $p \gg (\log N/N)^{1/(k-1)}$. We remark that this theorem was obtained by Bhattacharya, Ganguly, Shao and Zhao \cite{BGSZ} under the following conditions on $p$ and $\delta_N$: $p \to 0$, $\delta_N = O(1)$, $\delta_N^{-3}p^{k-2}(\log(1/p))^2 \to \infty$, and
\[
\min \{\delta_N p^k, \delta_N^2 p \} \ge N^{-\frac{1}{6(k-1)}}\log N
\]
For $k = 3$, the right-hand side can be relaxed to $N^{-1/6}(\log N)^{7/6}$ and for $k = 4$, it can be relaxed to $N^{-1/12}(\log N)^{13/12}$.  Our theorem extends the range of $p$ and $\delta_N$ for which the result is valid.

\subsection{Sidon equation}\label{sec:sidon}

We now turn to solutions of the Sidon equation $x + y = z + w$ in $[N]$.  We shall focus on solutions in which $x,y,z$ and $w$ are distinct and state our results in the context of the $4$-uniform hypergraph $\HH^{S}$ with vertex set on $[N]$ and edge $\{x,y,z,w\}$ if $x+y=z+w$.  We remark that it would be straightforward to extend our results to include solutions with a repeated element -- one may simply apply Proposition~\ref{prop:known} to the corresponding $3$-uniform hypergraph to bound its contribution.  The following lemma provides information on the average degree and the variance degree of $\HH^{S}$.

\begin{lemma}\label{lem:var_sidon}
Let $\HH^{S}$ be the hypergraph corresponding to pairwise distinct solutions of the Sidon equation $x + y = z + w$ in $[N]$. Then
\[
\bar{d}(\HH^{S}) = (1+o(1)) \frac{N^2}{3}
\]
and
\[
\sigma^2(\HH^S) = (1+o(1))\frac{N^4}{720}.
\]
\end{lemma}

\begin{proof}
We first compute, for each $a \in [N]$, the degree of $a$ in $\HH^S$. The degree of $a$ is one half times the number of solutions of $c + d  = a + b$ (we divide by two since permuting $c$ and $d$ gives the same solution in our context). We also allow for a moment solutions with repeated entries, since this only contributes with a $O(N)$ term to the degree. For each $b \in [N]$, there are
\[
\min \{ a+b-1, N\} - \max\{a+b-N,1\} + O(1)
\]
solutions of that equation. Therefore
\begin{align*}
2d_{\HH^S}(a) &= \sum_{b=1}^{N-a} (a+b) + \sum_{b=N-a+1}^{N} (2N - a - b) + O(N) \\
&= \frac{N^2}{2} + a(N-a) + O(N)
\end{align*}
and so $d_{\HH^S}(a) = N^2/4 + a(N-a)/2 + O(N)$. The first part of the lemma now follows by observing that
\[
\sum_{a =1}^{n} a(N-a) = \frac{N^3}{6} + O(N^2).
\]
Finally we have
\begin{align*}
N\sigma^2(\HH^S) &= \sum_{a = 1}^{N} \left( \frac{a(N-a)}{2} - \frac{N^2}{12} \right)^2 + O(N^4) \\
&= \frac{N^5}{144} - \frac{N^2}{12} \sum_{a =1}^{N} a(N-a) + \frac{1}{4}\sum_{a=1}^{N} a^2(N-a)^2 + O(N^4) \\
&= \frac{N^5}{144} - \frac{N^5}{72} + \frac{N^5}{120} + O(N^4) \\
&= \frac{N^5}{720} + O(N^4),
\end{align*}
where we used
\[
\sum_{a =1}^{n} a^2(N-a)^2 = \frac{N^5}{30} + O(N^4).
\]
\end{proof}

In the uniform model, we may obtain now the following result from Theorem \ref{thm:main}.

\begin{theorem}\label{thm:sidon_m}
Let $0 \le m \le N$ be such that $t = m/N \le 1/2$. Let $a_N$ be a sequence such that
\[
t^{7/2}N^{5/2} (\log N)^{1/2}\, \ll\, a_N \, \ll\, t^{17/4}N^3.
\]
Then
\[
\pr{D^{\HH^{S}}(B_m) \ge a_N}\, =\, \exp \left( \frac{-(360+o(1))a_N^2}{(1-t)t^{7}N^5} \right).
\]
\end{theorem}

We conclude obtaining the following result in the binomial model, where we use that $e(\HH^{S}) = (1+o(1))N^3/12$, by Lemma \ref{lem:var_sidon} and $\bar{d} = 4e(\HH^{S})/N$, since the hypergraph is $4$-uniform.

\begin{theorem}\label{thm:sidon_p}
Given a sequence $p = p_N$ which may be a constant in $(0,1/2)$ or converge to $0$, let $\delta_N$ be a sequence satisfying
\[
\sqrt{\frac{\log N}{pN}} \, \ll\, \delta_N\, \ll\, p^{1/4}\,.
\]
Then
\[
\pr{ D^{\HH^{S}}(B_p)\, \ge\, \delta_N p^4e(\HH^S)} \,= \, \exp \left( \frac{-(5+o(1))\delta_N^2pN}{162(1-p)} \right)\, .
\]
\end{theorem}

\bibliographystyle{plain}
\bibliography{references}

\end{document}